\documentclass[10pt,reqno]{amsart}

\usepackage{amsmath}
\usepackage{amsfonts}
\usepackage{amssymb}
\usepackage{bm}
\newcommand{\vertiii}[1]{{\left\vert\kern-0.25ex\left\vert\kern-0.25ex\left\vert #1 
    \right\vert\kern-0.25ex\right\vert\kern-0.25ex\right\vert}}
\usepackage{amsthm}
\usepackage{mathrsfs}
\usepackage{latexsym}

\usepackage{cite} 

\usepackage{esint}

\numberwithin{equation}{section}

\usepackage{graphicx,subfigure}
\usepackage{color}

\usepackage{ifthen} 

\provideboolean{shownotes} 
\setboolean{shownotes}{true} 
\newcommand{\margnote}[1]{
\ifthenelse{\boolean{shownotes}}%
{\marginpar{\raggedright\tiny\texttt{#1}}}%
{}%
}

\newcommand{\hole}[1]{
\ifthenelse{\boolean{shownotes}}%
{\begin{center} \fbox{ \rule {.25cm}{0cm}
\rule[-.1cm]{0cm}{.4cm} \parbox{.85\textwidth}{\begin{center}
\texttt{#1}\end{center}} \rule {.25cm}{0cm}}\end{center}}
{}
}

\usepackage[colorlinks=true,urlcolor=blue,
citecolor=red,linkcolor=blue,linktocpage,pdfpagelabels,
bookmarksnumbered,bookmarksopen]{hyperref}

\theoremstyle{plain}

\newtheorem{lemma}{Lemma}[section]
\newtheorem{theorem}[lemma]{Theorem}
\newtheorem{proposition}[lemma]{Proposition}
\newtheorem{corollary}[lemma]{Corollary}

\theoremstyle{definition}
\newtheorem{remark}[lemma]{Remark}
\newtheorem{definition}[lemma]{Definition}

\theoremstyle{remark}

\newcommand{\R}{\mathbb{R}}
\newcommand{\C}{\mathbb{C}}
\newcommand{\Z}{\mathbb{Z}}
\newcommand{\N}{\mathbb{N}}
\newcommand{\bbS}{\mathbb{S}}

\newcommand{\tiA}{\widetilde{A}}
\newcommand{\tiB}{\widetilde{B}}
\newcommand{\tiL}{\widetilde{L}}

\newcommand{\Or}{{\mathcal{O}}}

\newcommand{\cD}{{\mathcal{D}}}

\newcommand{\cS}{{\mathcal{S}}}

\newcommand{\cU}{{\mathcal{U}}}
\newcommand{\cE}{{\mathcal{E}}}
\newcommand{\cV}{{\mathcal{V}}}

\newcommand{\vep}{\varepsilon}

\renewcommand{\Re}{\mathrm{Re}\,} 
\renewcommand{\Im}{\mathrm{Im}\,}

\newcommand{\bu}{\overline{u}}

\newcommand{\bam}{\overline{m}}
\newcommand{\bbu}{\overline{\bm{u}}}
\newcommand{\bU}{\overline{U}}
\newcommand{\bV}{\overline{V}}
\newcommand{\bW}{\overline{W}}

\newcommand{\bL}{\overline{L}}

\newcommand{\bB}{\overline{B}}

\newcommand{\bA}{\overline{A}}

\newcommand{\bM}{\overline{M}}

\newcommand{\brho}{\overline{\rho}}
\newcommand{\bthe}{\overline{\theta}}

\newcommand{\bp}{\overline{p}}
\newcommand{\be}{\overline{e}}

\newcommand{\bet}{\overline{\eta}}

\newcommand{\hU}{\widehat{U}}
\newcommand{\hV}{\widehat{V}}
\newcommand{\hZ}{\widehat{Z}}

\newcommand{\hg}{\hat{g}}
\newcommand{\hh}{\hat{h}}

\newcommand{\<}{\langle}
\renewcommand{\>}{\rangle}

\begin{document}

\title[Dissipativity and decay for a radiation hydrodynamics system]{Dissipative structure and decay rate for an inviscid non-equilibrium radiation hydrodynamics system}



\author[C. Lattanzio]{Corrado Lattanzio}

\address[C. Lattanzio]{Dipartimento di Ingegneria e Scienze dell'Informazione e Matematica\\Universit\`a degli Studi dell'Aquila\\via Vetoio (snc), Coppito I-67010, L'Aquila (Italy)}

\email{corrado.lattanzio@univaq.it}

\author[R. G. Plaza]{Ram\'on G. Plaza}

\address[R. G. Plaza]{Departamento de Matem\'aticas y Mec\'anica\\Instituto de 
Investigaciones en Matem\'aticas Aplicadas y en Sistemas\\Universidad Nacional Aut\'onoma de 
M\'exico\\Circuito Escolar s/n, Ciudad Universitaria C.P. 04510 Cd. Mx. (Mexico)}

\email{plaza@aries.iimas.unam.mx}

\author[J. M. Valdovinos]{Jos\'{e} M. Valdovinos}
	
\address[J. M. Valdovinos]{Departamento de Matem\'aticas y Mec\'anica\\Instituto de 
Investigaciones en Matem\'aticas Aplicadas y en Sistemas\\Universidad Nacional Aut\'onoma de 
M\'exico\\ Circuito Escolar s/n, Ciudad Universitaria, C.P. 04510\\Cd. de M\'{e}xico (Mexico)}
\curraddr{\textsc{Institut de Math\'ematiques de Toulouse\\Universit\'e de Toulouse\\118, route de Narbonne, F-31062, Toulouse Cedex 9 (France)}}
	
\email{jvaldovi@math.univ-toulouse.fr}




\begin{abstract} 
This paper studies the diffusion approximation, non-equilibrium model of radiation hydrodynamics derived by Buet and Despr\'es (J. Quant. Spectrosc. Radiat. Transf. 85 (2004), no. 3-4, 385--418). The latter describes a non-relativistic inviscid fluid subject to a radiative field under the non-equilibrium hypothesis, that is, when the temperature of the fluid is different from the radiation temperature. It is shown that local solutions exist for the general system in several space dimensions. It is also proved that only the one-dimensional model is genuinely coupled in the sense of Kawashima and Shizuta (Hokkaido Math. J. 14 (1985), no. 2, 249--275). A notion of entropy function for non-conservative parabolic balance laws is also introduced. It is shown that the entropy identified by Buet and Despr\'es is an entropy function for the system in the latter sense. This entropy is used to recast the one-dimensional system in terms of a new set of perturbation variables and to symmetrize it. With the aid of genuine coupling and symmetrization, linear decay rates are obtained for the one dimensional problem. These estimates, combined with the local existence result, yield the global existence  and decay in time of perturbations of constant equilibrium states in one space dimension.
\end{abstract}

\maketitle



\section{Introduction}
\label{sec:intro}

The field of radiation hydrodynamics (cf. \cite{MiW-Mi84,Pomr73,Cas04}) is concerned with situations where (thermal) radiation effects are taken into account in the description of fluid motion. While at moderate temperatures the contribution of radiation to the dynamics of the fluid is by means of energy exchanges due to radiative processes, at high temperatures the thermal radiation may become comparable or even dominate the fluid state variables, and when this is the case the radiation significantly affects the dynamics of the fluid. Radiation hydrodynamics finds applications in various astrophysical phenomena (such as supernova explosions, the description of stellar winds, or nonlinear stellar pulsations; see, for example, Kippenhahn and Weigert \cite{KipWei90}), as well as in high-temperature plasma physics (cf. Zel{$'$}dovich and Raizer \cite{ZelRai02}). 

The most general system of equations describing the coupling of radiation and hydrodynamics is quite complicated to solve, both analytically and numerically. Radiation, for instance, is described by an assembly of photons, which are massless particles travelling at the speed of light $c$, and thus a description in the framework of special relativity is needed. The radiation is described by a transport equation for photon distribution with a non-local source term. Coupling the standard hydrodynamics equations for a gas and the radiative transfer equation results into a very complicated system which can be approximated in different physically valid regimes and, therefore, it is natural to consider reduced models. One of these approximation regimes is called the \emph{diffusion approximation} (also called the Eddington approximation), which quantifies the energy flow due to radiation in a semi-quantitative sense (cf. \cite{JiaZh18,Jia17}). This approximation is valid for optically thick fluids for which the photons emitted by the gas have a high probability of reabsorption. In most applications the fluid velocities are small compared to the velocity of light, so the flow description can be made through the classical Euler hydrodynamics system and by taking an approximation of order $O(v/c)$, where $v$ is the characteristic velocity of the fluid. It is important to keep terms of order $O(v/c)$ even when $|v/c|$ is small, in order to avoid neglecting the work done by radiation pressure and not to give rise to an incorrect radiation spectrum (see, e.g., Buchler \cite{Buch83}). Neglecting terms of order $O(v^2/c^2)$ is, on the other hand, consistent with using non-relativistic hydrodynamics equations for the material fluid. Under this point of view, Lowrie \emph{et al.} \cite{LoMoHi99} derived a set of equations describing radiation hydrodynamics to which one can apply Eulerian conservative high-order Godunov-type schemes commonly used in hydrodynamics \cite{Hirs07, Lev92}. The authors applied a simplified asymptotic analysis expansion, very similar to a Chapman-Enskog or Hilbert expansion, and considered the equilibrium diffusion model, that is, when the matter temperature is taken a priori equal to the radiation temperature. In a later work and in the same spirit of Lowrie \emph{et al.}, Buet and Despr\'es \cite{BuDe04} performed the same asymptotic analysis in the diffusion approximation regime in order to derive for the first time the \emph{non-equilibrium diffusion model}, that is, when the temperature of the fluid, $\theta$, is different from the temperature of radiation, $\theta_{r}$. This is the physical model that we address in the present paper. Notably, this limiting model has been rigorously justified in recent works, both in $\R^3$ \cite{LiZha24} and in a three-dimensional torus \cite{JLZ25}.

The non-equilibrium diffusion radiation hydrodynamics model derived by Buet and Despr\'es (see system (67) in \cite{BuDe04}) reads 
\begin{equation}\label{neq-rad-hyd}
\begin{aligned}
\partial_{t}\rho + \nabla \cdot (\rho \bm{u})&=0, \\
\partial_{t}(\rho \bm{u}) + \nabla \cdot \big( \rho \bm{u} \otimes \bm{u} \big) + \nabla \big( p +\tfrac{1}{3}\eta\big) &=0, \\ 
\partial_{t}\big( \rho E  + \eta \big) + \nabla \cdot \big( (\rho E+ \eta )\bm{u} + \big(p +\tfrac{1}{3}\eta \big)\bm{u} \big)&= \nabla \cdot \big( \tfrac{1}{3\sigma_{s}}\nabla \eta \big), \\
\partial_{t}\eta + \nabla \cdot (\eta \bm{u} ) + \tfrac{1}{3}\eta \nabla \cdot \bm{u} &= \nabla \cdot \big( \tfrac{1}{3\sigma_{s}}\nabla \eta \big) + \sigma_{a} (\theta^{4} -\eta). \
\end{aligned}
\end{equation}
Here $t>0$ denotes time, $x \in \R^d$ denote space variables, with $d = 1,2,3$, and $\nabla$ is the space gradient operator. The unknowns are the density $\rho$, the velocity field $\bm{u} \in \R^d$, the absolute temperature $\theta$ of the fluid and the energy of radiation $\eta$ (which is equal to $\theta_{r}^4$ and depends on the radiation intensity). As usual, $p$ is the thermodynamic pressure and $E=e+\vert \bm{u} \vert^{2}/2$ is the total energy of the fluid, with $e$ the internal energy (per unit mass). The pressure and the internal energy are (smooth) function of the independent thermodynamic variables $\rho$ and $\theta$, that is, $p=p(\rho, \theta)$ and $e=e(\rho, \theta)$. The absorption coefficient $\sigma_{a}$ and the scattering coefficient $\sigma_{s}$ are assumed to be positive constants (for simplicity, the analysis by Buet and Despr\'es is made under the \emph{gray hypothesis}, that is, constant emissivity across wavelengths). The first three equations are the usual balance laws describing an inviscid, non-heat-conducting compressible fluid in which the momentum and energy equations have been modified accordingly to account for the effect of the radiation energy $\eta$. System \eqref{neq-rad-hyd} is a non-conservative system of hyperbolic-parabolic type, where the only parabolic term is proportional to $\Delta \eta$, due to the inviscid and non-heat-conducting nature of the fluid. Also notice that there are no damping terms, that the only balance term appears in the radiation equation and that for a constant state $(\rho, \bar{\bm{u}}, \bar{E}, \bar{\eta})$ to be a solution it must satisfy $\bar{\eta} = \bar{\theta}^4$, an equation that defines the equilibrium manifold. System \eqref{neq-rad-hyd} is often referred to as the non-equilibrium-diffusion limit \cite{MiW-Mi84,Pomr73}.

\subsection{Previous works on the non-equilibrium model}

The most studied  model in the context of radiation hydrodynamics is obtained from  \eqref{neq-rad-hyd} by neglecting the time derivative for the energy of radiation which, as a consequence, is time-asymptotically at equilibrium. Thus, the resulting system is an hyperbolic-elliptic coupled system, and it is also referred to as a non-equilibrium model because the gas is not in thermodynamical equilibrium.
There is a vast mathematical literature concerning these hyperbolic-elliptic coupled models which we will not review here; for an abridged list of references see \cite{BuDe04,LCG06,LiuKaw11,LMS07,LCG07,WaXi11a,LinC11,WnWn09,WaXi11b,WnXi11,LinGou11,KN01,KaNN99,KaNN03,LMNPZ09}. In contrast, the evolution non-equilibrium model \eqref{neq-rad-hyd} has been less analyzed. In most cases, the model is further endowed with physical fluid viscosities \cite{Jia17,WaXiYa23,KiHoKi23,JianXieZ09}, or with damping terms \cite{BlDuNec16,BlDuNec22}. For instance, Jiang \emph{et al.} \cite{JianXieZ09} showed the global (in time) existence and uniqueness of solution for the one-dimensional initial-boundary value problem of a viscous and heat-conducting fluid, for suitable smooth initial data and when the heat conductivity satisfies a physical growth condition with respect to the fluid temperature. The model studied by Jiang \emph{et al.} has been subsequently studied in other works. Jiang \cite{Jia17} established the global well-posedness in Sobolev spaces for the Cauchy problem in the perturbation framework in the case of the $3d$-polytropic ideal gas system, and obtained the convergence rate $(1+t)^{-3/4}$ for the $H^{3}(\R^3)-$norm of solutions when the initial data belongs to $L^1(\R^3)$. Similar results have been obtained by Wang \emph{et al.} \cite{WaXiYa23}, where the authors have used Littlewood-Paley decomposition. Later Kim \emph{et al.} \cite{KiHoKi23} extended the previous results to more general fluids, and have also obtained the decay rate $(1+t)^{-s/2}$ of the $H^{N}(\R^3)-$norm ($N\geq 3$) of the solutions when the initial data belongs to the homogeneous negative Sobolev space $\dot{H}^{-s}(\R^3)$, for $s\in [0,1/2]$. We also mention the work by Jiang and Zhou \cite{JiaZh18}, where they have established the local well-posedness of smooth solutions for the $3d$ polytropic gas case of the same model studied by Jiang \emph{et al.} \cite{JianXieZ09}, in which viscous and heat-conduction effect for the fluid are ignored and the radiation pressure term is considered but only for the momentum balance equation, that is, a term proportional to $\nabla \big( \tfrac{1}{3}\eta  \big)$ is incorporated. 

Regarding the models with damping, Blanc \emph{et al.} \cite{BlDuNec16} established the global in time existence of solutions to the Cauchy problem for the system \eqref{neq-rad-hyd} in three dimensions with damping and heat conduction effects for the fluid, provided the initial data is a small perturbation of a constant equilibrium state. The authors in \cite{BlDuNec16} have also obtained the same results for the equilibrium-diffusion limit, that is, system \eqref{neq-rad-hyd} where $\eta$ is replaced by $\theta^{4}$ and there is no equation for the energy of radiation.

To sum up, although there are several works that studied the well-posedness of the non-equilibrium-diffusion limit system, most of them do not consider the original inviscid, non-heat-conducting and non-damped system \eqref{neq-rad-hyd} derived by Buet and Depr\'es \cite{BuDe04}.

\subsection{The contributions of this work}

In the present work we are interested in the global (in time) existence of solutions for the Cauchy problem of system \eqref{neq-rad-hyd} with initial data being a small perturbation of a constant equilibrium state, and study their asymptotic behaviour. In our study, we combine local (in time) existence of solutions results with a priori energy estimates and (nonlinear) decay rates of the local solution, via an continuation argument to get the global existence and asymptotic behaviour of solutions. For the local existence of solutions we apply a result by Kawashima \cite{KaTh83} for hyperbolic-parabolic system of composite type. For obtaining the a priori energy estimates as well as the decay rates, we employ a technique developed in the same work of Kawashima \cite{KaTh83} and that consists in studying the \textit{dissipative structure} of the linear system around the constant equilibrium state.

The dissipative structure mentioned above refers to the fact that the system satisfies any of the properties stated in Theorem \ref{Eq-Th} below. As those properties are equivalent, the genuine coupling condition, which is straightforward to verify, tells us that in order to posses this structure, the system needs to have enough dissipation due to ``relaxation'' (zero-order space derivatives terms) and ``viscosity'' (second-order space derivatives) mechanisms. For the linear version around a constant equilibrium state of system \eqref{neq-rad-hyd},  it turns out that the one-dimensional space system has enough dissipation, but this fails to be the case when the space dimension is $d \geq 2$ (see Appendix \ref{non-gen-cou-md}). Thus, in the present work we consider the one-dimensional case because of this technicality. 

It should be noted that our results can be easily extended to the several space dimensions case if we take into consideration damping, as it has been done in \cite{BlDuNec16}, or viscosity effects for the fluid in system \eqref{neq-rad-hyd}. In this sense, our results complement those of Blanc \emph{et al.} \cite{BlDuNec16}, as we also provide decay rates for the solutions. Notably, the authors use the compensating matrix $K$, given by the Equivalence Theorem \ref{Eq-Th}, for performing some energy estimates but not for obtaining the decay rates, as we do it in the present work. In addition, although one might think that the one dimensional case is easier to handle, this is not the case at the time of performing the nonlinear energy estimate for getting the a priori energy estimates and the decay rates. In order to do so, the system has to be written in a very specific way and this is related to the existence of an entropy function/flux entropy pair (see \cite[Chapter IV]{KaTh83} or \cite[Section 7]{KY09}), which is the case for the system under consideration and that we explain in detail below; see Section \ref{subsec:yong}. It turns out that the entropy that works for system \eqref{neq-rad-hyd} is the classical one for the fluid plus the entropy associated to radiation and which was proposed by Buet and Despr\'es; see \cite[Corollary 2]{BuDe04},  or equation \eqref{form-Ent.} below. Up to the authors' knowledge this entropy structure of system \eqref{neq-rad-hyd} has not been reported in previous works. 

\subsection{Equations and assumptions}

In this paper we consider the non-equilibrium system of equations \eqref{neq-rad-hyd} describing a non-relativistic inviscid fluid under the effects of radiation. It is a non-conservative system in Eulerian variables and in several space dimensions. Substituting the equation for the energy of the radiation into the energy equation and performing some straightforward algebra, it is possible to recast system \eqref{neq-rad-hyd}  as the following equivalent system (details are left to the reader):

\newpage

\begin{equation}
\label{neq-rad-hyd-var}
\begin{aligned}
\partial_{t}\rho + \nabla \cdot (\rho \bm{u})&=0, \\
\partial_{t}(\rho \bm{u}) + \nabla \cdot \big( \rho \bm{u} \otimes \bm{u} \big) + \nabla \big( p +\tfrac{1}{3}\eta\big) &=0, \\ 
\partial_{t}(\rho E) + \nabla \cdot \big( \rho E \bm{u} + (p+\tfrac{1}{3}\eta)\bm{u} \big) &= - \sigma_{a} (\theta^{4} -\eta)+ \tfrac{1}{3}\eta \nabla \cdot\bm{u}, \\
\partial_{t}\eta + \nabla \cdot (\eta \bm{u} )  &= \nabla \cdot \big( \tfrac{1}{3\sigma_{s}}\nabla \eta \big) + \sigma_{a} (\theta^{4} -\eta)- \tfrac{1}{3}\eta \nabla \cdot \bm{u}. \
\end{aligned}
\end{equation}

In the forthcoming analysis, system \eqref{neq-rad-hyd-var} will often appear more suitable for our needs and we shall be working with both variants of the physical model almost without distinction.

Let us state the physical assumptions for system \eqref{neq-rad-hyd} (or equivalently, for system \eqref{neq-rad-hyd-var}) under consideration in this paper.
\begin{itemize}
\item[(H$_1$)] \phantomsection
\label{H1} The independent thermodynamic variables are the density $\rho >0$ and the absolute temperature $\theta>0$. They take values in convex open set
\[
\cD =\{ (\rho, \theta):\, \rho>0, \quad \theta>0 \}.
\]
\item[(H$_2$)] \phantomsection
\label{H2} The thermodynamic pressure $p$, the internal energy (per unit mass) of the fluid $e$, and the specific entropy of the fluid $s$ are smooth functions of $\rho$ and $\theta$, that is $p$, $e$, and $s\in C^{\infty}(\cD)$. They satisfy
\begin{equation}
\label{therm-prop}
p>0,\quad p_{\rho}>0,\quad p_{\theta}>0,\quad e_{\theta}>0,
\end{equation}
as well as the volumetric First Law of Thermodynamics
\[
d e=\theta\, ds -p d\left(\frac{1}{\rho} \right),
\]
which implies the relations 
\begin{equation}\label{therm-rel}
e_{\rho}=(p-\theta p_{\theta})/\rho^{2},\quad s_{\rho}=-p_{\theta}/\rho^{2},\quad s_{\theta}=e_{\theta}/\theta.
\end{equation}
\item[(H$_3$)] \phantomsection
\label{H3} The absorption coefficient $\sigma_{a}$ and the scattering coefficiente $\sigma_{s}$ are positive constants.  
\end{itemize}

\begin{remark}
\label{thrm-rmrk}
Notice that, for convenience, we have chosen $\rho$ and $\theta$ as the independent thermodynamic variables. In addition, it is to be observed that the hypotheses \eqref{therm-prop} on the thermodynamic potentials $p$ and $e$ are quite general and satisfy the conditions for an \emph{arbitrary Weyl fluid} \cite{We49}, namely, a generalized Gay-Lussac's law ($p_\theta > 0$), adiabatic increase of pressure effects compression ($p_\rho > 0$) and the increase of internal energy due to an increase of temperature at constant volume ($e_\theta > 0$). A typical example is that of an ideal gas satisfying
\[
p(\rho, \theta) = R \rho \theta, \qquad e(\rho, \theta) = \frac{R  \theta}{\gamma -1},
\]
where $R > 0$ is the universal gas constant and $\gamma > 1$ is the adiabatic exponent. In radiation hydrodynamics one may consider other types of potentials which fall under the category of Weyl; see \cite{MiW-Mi84,Pomr73,Cas04} for further information.
%
\end{remark}

As we have already mentioned, we specialize our stability analysis to the case of one space dimension ($d = 1$). Hence, let us write the one-dimensional version of system \eqref{neq-rad-hyd}:
\begin{equation}
\label{neq-rad-hyd-1d}
\begin{aligned}
\partial_{t}\rho + \partial_{x} (\rho u)&=0, \\
\partial_{t}(\rho u) + \partial_{x} \big( \rho u^2 + p +\tfrac{1}{3}\eta \big) &=0, \\ 
\partial_{t}\big( \rho E  + \eta \big) + \partial_{x} \big( \big( \rho E+ \eta +p +\tfrac{1}{3}\eta \big)u \big)&= \tfrac{1}{3\sigma_{s}}\partial_{xx}\eta, \\
\partial_{t}\eta + \partial_x(\eta u ) + \tfrac{1}{3}\eta \partial_{x} u &= \tfrac{1}{3\sigma_{s}}\partial_{xx}\eta  + \sigma_{a} (\theta^{4} -\eta).
\end{aligned}
\end{equation}

%
%
%
%

\subsection{Main result}
%

Let us now state the main result of the paper.

\begin{theorem}[global existence and time asymptotic decay]
\label{thmgloex} 
Suppose hypotheses \hyperref[H1]{\rm{(H$_1$)}} - \hyperref[H3]{\rm{(H$_3$)}} hold. Let $\bV=(\brho, \bu, \bthe, \bet) \in \R^4$ be a constant equilibrium state satisfying $\brho, \bthe, \bet > 0$, $\bthe^4 = \bet$. Assume that $V_0 - \bV \in \big(H^{s}(\R) \cap L^{1}(\R)\big)^4$ with $s\geq 3$. Then there exists a positive constant $\vep > 0$ such that if
\begin{equation}
\label{lessep}
\Vert V_0 - \bV \Vert_s + \Vert V_0 - \bV \Vert_{L^1} \leq \vep,
\end{equation}
then the Cauchy problem for system \eqref{neq-rad-hyd-1d} with initial condition $V(0) = V_0$ has a unique global solution $V(x,t)= (\rho, u, \theta,\eta)(x,t)$ satisfying 
\[
\begin{aligned}
&\rho-\brho, u-\bu, \theta-\bthe \in C\left( (0,\infty); H^{s}(\R) \right)    \cap C^{1}\left( (0,\infty); H^{s-1}(\R) \right),\\
& \eta-\bet \in C\left( (0,\infty); H^{s}(\R) \right) \cap C^{1}\left( (0,\infty); H^{s-2}(\R) \right).
\end{aligned}
\]
Moreover, the solution satisfies 
\begin{equation}
\label{fin-apr-ee-global}
\sup_{0 \leq \tau \leq t} \Vert (V-\bV)(\tau) \Vert_{s}^{2} + \int_{0}^{t} \Vert \partial_{x}(\rho, u, \theta)(\tau) \Vert_{s-1}^{2} + \Vert \partial_{x}\eta(\tau) \Vert_{s}^{2}\, d \tau \leq C \Vert V_{0} - \bV \Vert_{s,1}^{2},
\end{equation}
as well as
\begin{equation}
\label{nl-dec-global}
\Vert (V-\bV)(t) \Vert_{s-1} \leq C (1+t)^{-1/4} \Big( \Vert V_{0} - \bV \Vert_{s-1} + \Vert V_{0} - \bV \Vert_{L1} \Big),
\end{equation}
for all $t \geq 0$ and for some uniform constant $C > 0$.
\end{theorem}

\begin{remark}
It is to be observed that Theorem \ref{thmgloex} establishes the global well-posedness of classical solutions (thanks to the Sobolev embedding) to the non-equilibrium diffusion limit system \eqref{neq-rad-hyd} in the perturbation framework and in one space dimension. The constant state is supposed to belong to the equilibrium manifold (that is, it satisfies $\bthe^4 = \bet$), although the system itself is in the non-equilibrium regime. The initial (perturbation) conditions are assumed to have finite energy ($V_0 - \bV \in H^{s}(\R)^4$, $s \geq 3$) and finite mass ($V_0 - \bV \in L^{1}(\R)^4$). If the initial perturbations are sufficiently small then global classical solutions exist and decay as $t \to \infty$. Theorem \ref{thmgloex} also determines rates of decay for the perturbation variables of algebraic type.
\end{remark}


\subsection*{Plan of the paper}

In Section \ref{sec:local} we recall the notions of strict dissipativity and genuine coupling, and state the local existence result and their consequences (local energy estimates). In Section \ref{subsec:yong} we extend the definition of entropy to the case of non-conservative viscous balance laws, symmetrize the system and define the new perturbation variables. Section \ref{sec:linear} examines the linear dissipative structure of the system in the new variables and establishes the desired linear decay rates. The last Section \ref{secnonlinear} is devoted to closing the nonlinear energy estimates and to the proof of the main Theorem \ref{thmgloex}. We make some final comments in the dicussion Section \ref{secdicussion}. In addition, we have included three Appendices. Appendix \ref{app1} states the local existence theorem in several space dimensions. In Appendix \ref{non-gen-cou-md} we prove that the system is not genuinely coupled when the dimension is $d \geq 2$. Appendix \ref{sptr-anal} contains a spectral analysis of the generated semigroup, which is needed for some of the energy estimates in Section \ref{sec:linear}.

\subsection*{Notations} 
Standard Sobolev spaces of functions on the real line will be denoted as $H^s(\R)$, with $s \in \R$, endowed with the standard inner products and norms. The norm on $H^s(\R)$ will be denoted as $\| \cdot \|_s$ and the norm in $L^2$ will be henceforth denoted by $\| \cdot \|_0$. Any other $L^p$ -norm will be denoted as $\| \cdot \|_{L^p}$ for $p \geq 1$. We denote the real and imaginary parts of a complex number $\lambda \in \C$ by $\Re\lambda$ and $\Im\lambda$, respectively. For any $\alpha \in \R$ we denote as $\lfloor \alpha \rfloor \in \Z$ the integer satisfying $\lfloor \alpha \rfloor \leq \alpha < \lfloor \alpha \rfloor+1$. $0_{p \times q}$ will denote the zero $p \times q$ block matrix, for any $p, q \in \N$. The square zero and identity $p \times p$ matrices, with $p \in \N$, will be written as $0_{p}$ and $I_{p}$, respectively. The canonical (row) basis in $\R^{1\times n}$ is denoted by $\hat{e}_j$, $1 \leq j \leq n$.



\section{Preliminaries}
\label{sec:local}

In this section we present some preliminary material. Section \ref{secequiv} contains the definition of strict dissipativity and the Equivalence Theorem by Shizuta and Kawashima \cite{ShKa85} for second order systems (see Theorem \ref{Eq-Th} below). In Section \ref{sec:localex} we state the local (in-time) existence of solutions to the Cauchy problem for system \eqref{neq-rad-hyd} in one space dimension and with initial data close to a constant equilibrium state. This result is a consequence of a more general theorem by Kawashima \cite{KaTh83} (see Theorem \ref{loc-ex-th} below). In Section \ref{secapriori} we present some \textit{a priori} energy estimates for the local solutions.
%

\subsection{Strict dissipativity and the Equivalence Theorem}
\label{secequiv}

Consider a linear second order system in any space dimension $d\geq 1$ of the form
\begin{equation}\label{gen-lin-sys-d}
A^0U_{t}+\sum_{j=1}^{d}A^{j}U_{x_j}+LV =\sum_{j,k=1}^{d}B^{jk}U_{x_{j}x_{k}},
\end{equation}
where $U=U(x,t)\in \R^{n}$ is the vector of state variables, and $x\in \R^d$ and $t\geq 0$ represent space and time, respectively. The coefficients $A^0$, $A^j$, $L$ and $B^{jk}$, $j,k=1,\ldots, d$, are constant square real matrices of order $n$ satisfying:
\begin{itemize}
\item[(A$_1$)] \phantomsection
\label{A1} $A^0$, $A^j$, $j=1,\ldots,d$, are symmetric with $A^0$ positive definite ($A^0>0$);
\item[(A$_2$)] \phantomsection
\label{A2} $L$, $B^{jk}$, $j,k=1,\ldots,d$, are symmetric with $L$ positive semi-definite ($L \geq 0$) and the symbol $\sum_{j,k=1}^{d}\omega_{j}\omega_{k}B^{jk}$ is positive semi-definite for all $\omega\in\bbS^{d-1}$.
\end{itemize}

Apply the Fourier transform to system \eqref{gen-lin-sys-d} in order to obtain
\begin{equation}\label{gen-lin-sys-d-F}
A^0\hU_{t}+\big( i \vert \xi \vert A(\omega) + L + \vert \xi \vert^2 B(\omega) \big) \hU=0,
\end{equation}
where
\[
A(\omega):= \sum_{j=1}^{d}\omega_{j}A^j, \qquad B(\omega):= \sum_{j,k=1}^{d}\omega_{j}\omega_{k}B^{jk},
\]
for $\omega=\xi/\vert \xi \vert$, $\xi \in \R^d$, $\xi \neq 0$. The solutions to this linear system are related to the eigenvalue problem
\begin{equation}\label{gen-lin-sys-d-eig}
\big(\lambda A^0+ i\vert \xi \vert A(\omega) + L + \vert \xi \vert^{2} B(\omega) \big)v=0,
\end{equation}
with $\lambda=\lambda(\xi)\in \C$, $v=v(\xi)\in\R^n$, $v\neq 0$. Under these circumstances, we have the following Equivalence Theorem by Shizuta and Kawashima \cite{ShKa85}.

\begin{theorem}[equivalence theorem \cite{ShKa85}]
\label{Eq-Th}
Under assumptions \hyperref[A1]{\rm{(A$_1$)}} -- \hyperref[A2]{\rm{(A$_2$)}}, the following statements are equivalent:
\begin{itemize}
\item[(i)] For each $\xi\in\R^d$, $\xi\neq 0$, the real part of $\lambda=\lambda(\xi)$, solutions of \eqref{gen-lin-sys-d-eig}, satisfy $
\Re \, \lambda <0$.
\item[(ii)] For any $\psi\in\R^{n}$, $\psi \neq 0$, satisfying $B(\omega)\psi=L\psi=0$ for some $\omega\in \bbS^{d-1}$, there holds
\[
\mu A^0\psi + A(\omega)\psi \neq 0,\,\, \mbox{for any}\,\, \mu\in\R.
\]
\item[(iii)] There exists a smooth real matrix valued function $K(\omega)$, $\omega\in\bbS^{d-1}$,
\[
\bbS^{d-1}\ni \omega \longmapsto K(\omega)\in M_{n}(\R)
\]
such that
\begin{itemize}
\item[(a)] $K(\omega)A^0$ is skew-symmetric for all $\omega\in\bbS^{d-1}$, and
\item[(b)] $[K(\omega)A(\omega)]^{s}+B(\omega)+L>0$ for all $\omega\in \bbS^{d-1}$,
\end{itemize}
where $[M]^{s}$ denotes the symmetric part of the matrix $M$, that is $[M]^s = \tfrac{1}{2}(M+M^{\top})$.
\item[(iv)] There exists a positive constant $c$ such that there holds
\[
\Re \lambda \leq - \frac{ c \vert \xi \vert^2}{1+\vert \xi \vert^2},
\]
for $\lambda= \lambda(\xi)$ solution of \eqref{gen-lin-sys-d-eig}, for all $\xi \in\R^d$, $\xi\neq 0$.
\end{itemize}
\end{theorem}

When the system \eqref{gen-lin-sys-d} satisfies (i) in Theorem \ref{Eq-Th} we say that the system is \textit{strictly dissipative}, and if (ii) holds true we say that the system satisfies the \textit{genuine coupling} condition or to be \textit{genuinely coupled}. The matrix $K$ in Theorem \ref{Eq-Th} (iii) is known as a \textit{compensating matrix} for the system \eqref{gen-lin-sys-d}.

\begin{remark}
The seminal work by Kawashima and Shizuta \cite{KaSh88a,ShKa85} established the conditions for the strict dissipativity of a large number of second order systems (such as linearizations around constant states of the Navier-Stokes and the Navier-Stokes-Fourier models). Thanks to the Equivalence Theorem \ref{Eq-Th}, the genuine coupling condition (which is an algebraic condition) implies the existence of compensating matrix functions which are very useful to obtain energy decay estimates at the linear level. These estimates can be applied to the nonlinear problem around constant states or even around small amplitude shock profiles.
\end{remark}


\subsection{Local existence}
\label{sec:localex}

Consider the one-dimensional non-equilibrium diffusion radiation  hydrodynamics system \eqref{neq-rad-hyd-1d} and define the state variables
\[
V(x,t):= (\rho, u, \theta, \eta)(x,t),
\]
which belong to the set
\begin{equation}\label{cV}
\cV :=\left\lbrace (\rho, u,\theta, \eta)\in \R^{4}: \rho \geq C_{1},\, \theta \geq C_{2},\, \eta\geq C_{3} \right\rbrace,
\end{equation}
for some fixed (but arbitrary) positive constants $C_{i}$, $i=1,2,3$.
We now pose the initial value problem of system \eqref{neq-rad-hyd-1d} with initial data
\begin{equation}\label{in-dat-1d}
V(x,0)=V_{0}(x)=(\rho_{0},u_{0},\theta_{0},\eta_{0})(x).
\end{equation}
For system \eqref{neq-rad-hyd-1d}, the \textit{equilibrium manifold} is defined by
\[
\cU_{\mathrm{eq}} := \left\lbrace V\in \cV: \eta=\theta^4  \right\rbrace;
\]
for more details, see Section \ref{subsec:yong}. Let us fix a constant equilibrium state $\bV=(\brho, \bu, \bthe, \bet) \in \cU_{\mathrm{eq}}$ of system \eqref{neq-rad-hyd-1d}, satisfying thus the relation
\[
\bet = \bthe^4.
\]
Then the local existence of perturbations of such equilibrium state is given by the next result.
\begin{theorem}\label{loc-ex-th-1d}
Let $\bV=(\brho, \bu, \bthe, \bet) \in \cU_{\mathrm{eq}}$ be a constant equilibrium state for system \eqref{neq-rad-hyd-1d}. Consider the initial value problem \eqref{neq-rad-hyd-1d} and \eqref{in-dat-1d} with $V_{0}-\bV \in H^{s}(\R)^4$, $s\geq 2$. Then there exists $\epsilon > 0 $ such that if
\[
a:= 
\Vert V_{0}-\bV \Vert_{s} \leq \epsilon,
\]
then we have that for $m_{1}\leq \rho_{0}(x) \leq M_{1}$, $m_{2}\leq \theta_{0}(x) \leq M_{2}$, $m_{3}\leq \eta_{0}(x) \leq M_{3}$ for all $x\in \R$ and positive constants $0 < m_{i} < M_{i}$, $i=1,2,3$, and for some $T_{0}=T_{0}(a)$, there exists a unique solution $V(x,t)$ to the initial value problem satisfying
\begin{equation}\label{reg-loc-ex-1d}
\begin{aligned}
&\rho-\brho, u-\bu, \theta-\bthe \in C\left( [0,T_{0}]; H^{s}(\R) \right)    \cap C^{1}\left( [0,T_{0}]; H^{s-1}(\R) \right),\\
& \eta-\bet \in C\left( [0,T_{0}]; H^{s}(\R) \right) \cap C^{1}\left( [0,T_{0}]; H^{s-2}(\R) \right).
\end{aligned}
\end{equation}
Moreover, the following estimate
\begin{equation}\label{loc-ex-ee-1d}
\sup_{0 \leq \tau \leq t} \Vert (V-\bV)(\tau) \Vert_{s}^{2} + \int_{0}^{t}  \Vert \partial_{x}(\rho, u, \theta)(\tau)\Vert_{s-1}^{2}+ \Vert \partial_{x} \eta(\tau) \Vert_{s}^{2}\, d \tau \leq C_{0} \Vert V_{0}-\bV \Vert_{s}^{2}
\end{equation}
holds for all $t\in [0,T_{0}]$ and for some positive constant $C_{0}$ depending on $\Vert V_{0}- \bV \Vert_{s}$.
\end{theorem}

\begin{remark}\label{cld-loc-sol-rm}
It is to be observed that, since $s\geq 2$, the continuous embedding $H^s(\R)\subset L^{\infty}(\R)$ implies that for sufficiently small initial data the solution $V(x,t)=(\rho, u, \theta, \eta)(x,t)$ will remain close to the constant equilibrium state $\bV=(\brho,\bu,\bthe, \bet)$ and will be such that
\[
\bam_{1}\leq \rho(x,t) \leq \bM_{1},\,\, \bam_{2}\leq \theta(x,t) \leq \bM_{2},\,\, \bam_{3}\leq \eta
(x,t) \leq \bM_{3},
\]
for all $x\in \R$ and $t\in [0,T_{0}]$, and some uniform positive constant $\bam_{i}$ and $\bM_{i}$, $i=1,2,3$, so that $V(x,t)\in \cV$, with $\cV$ as in \eqref{cV} for some positive constants $C_{i}=C_{i}(a)$, $i=1,2,3$. 
\end{remark}

Motivated by the energy estimate \eqref{loc-ex-ee-1d}, for any local solution on a time interval $[0,T]$ of the Cauchy problem \eqref{neq-rad-hyd-1d} and \eqref{in-dat-1d}, we define
\begin{equation}\label{Es}
E_{s}(t):= \sup_{0 \leq \tau \leq t}\Vert (V- \bV)(\tau) \Vert_{s}^{2},    
\end{equation}
\begin{equation}\label{Fs}
F_{s}(t):= \int_{0}^{t} \Vert \partial_{x}(\rho, u, \theta)(\tau) \Vert_{s-1}^{2} + \Vert \partial_{x}\eta (\tau) \Vert_{s}^{2}\, d \tau,
\end{equation}
as well as 
\begin{equation}\label{Ns}
N_{s}(t)^{2}:= E_{s}(t) + F_{s}(t), 
\end{equation}
for any $t\in [0,T]$.

\subsection{A priori energy estimates}
\label{secapriori}

Next, we enunciate some a priori energy estimates for the local solution of the Cauchy problem of system \eqref{neq-rad-hyd-1d} with initial data $V_{0}(x)$ close to a constant equilibrium state $\bV=(\brho, \bu, \bthe, \bet) \in \cU_{\mathrm{eq}}$. As we have already mentioned, we invoke a general result by Kawashima \cite{KaTh83} in order to obtain such estimates. For that purpose, we verify that the linearized system around the constant equilibrium state $\bV$ satisfies the genuine coupling condition. 

We start by writing the linear system around $\bV$, which is just the one-dimensional version of the linear system \eqref{lin-sys-md} in Appendix \ref{non-gen-cou-md}, and it reads
\begin{equation}\label{lin-sys-1d}
A^{0}V_{t}+A^{1}V_{x}+LV = BV_{xx},
\end{equation}
where the constant matrix coefficients are given by
\[
A^{0}=\begin{pmatrix}
1 & 0 & 0 & 0 \\ 0 & \brho & 0 & 0 \\ 0 & 0 & \brho\, \be_{\theta} & 0 \\ 0 & 0 & 0 & 1
\end{pmatrix}, \quad L = \begin{pmatrix}
0 & 0 & 0 & 0 \\ 0 & 0 & 0 & 0 \\ 0 & 0 & 4\sigma_{a}\bthe^3 & -\sigma_{a} \\ 0 & 0 & -4\sigma_{a}\bthe^3 & \sigma_{a}
\end{pmatrix},
\]
\[
A^{1}=\begin{pmatrix}
\bu & \brho & 0  & 0 \\ \bp_{\rho} & \brho\, \bu & \bp_{\theta} & 1/3 \\ 0 & \bthe\,\bp_{\theta} & \brho\,\bu\,\be_{\theta} & 0 \\ 0 & 4\bet/3 & 0 & \bu
\end{pmatrix}, \quad B =\begin{pmatrix} 0 & 0 & 0 & 0 \\ 0 & 0 & 0 & 0 \\ 0 & 0 & 0 & 0 \\ 0 & 0 & 0 & \frac{1}{3\sigma_{s}}
\end{pmatrix}.
\]
We use the symmetrizer
\[
S= \begin{pmatrix}
\bthe\,\bp_{\rho}\bet/\brho & 0 & 0 & 0 \\ 0 & \bthe\,\bet & 0 & 0 \\ 0 & 0 & \bet & 0 \\ 0 & 0 & 0 & \bthe/4
\end{pmatrix},
\]
to get the next symmetric version of system \eqref{lin-sys-1d},
\begin{equation}\label{1d-quas-symm}
\bA^{0}V_{t}+\bA^{1}V_{x}+\bL V = \bB V_{xx}, 
\end{equation}
with the constant matrix coefficients having the form
\[
\bA^{0}:=SA^{0}= \begin{pmatrix}
\bthe\, \bp_{\rho}\,\bet/\brho & 0 & 0 & 0 \\ 0 & \brho\,\bthe\,\bet & 0 & 0 \\ 0 & 0 & \brho\,\be_{\theta}\bet & 0 \\ 0 & 0 & 0 & \bthe/4
\end{pmatrix},
\]
\[
\bL:=SL= \begin{pmatrix}
0 & 0 & 0 & 0 \\ 0 & 0 & 0 & 0 \\ 0 & 0 & 4 \sigma_{a}\, \bthe^{3}\,\bet & -\sigma_{a}\bet \\ 0 & 0 & -\sigma_{a}\, \bet & \sigma_{a}\,\bthe/4
\end{pmatrix},
\]
\[
\bA^{1}:=SA^{1}= \begin{pmatrix}
\bu\,\bthe\,\bet\,\bp_{\rho}/\brho & \bthe\,\bet\,\bp_{\rho} & 0 & 0 \\ \bthe\,\bet\,\bp_{\rho} & \brho\, \bu\,\bthe\,\bet & \bthe\,\bet\,\bp_{\theta} & \bthe\,\bet/3 \\ 0 & \bthe\,\bet\,\bp_{\theta} & \brho\,\bu\,\bet\,\be_{\theta} & 0 \\ 0 & \bthe\,\bet/3 & 0 & \bu\,\bthe/4
\end{pmatrix}, \,\,
\bB:= SB= \begin{pmatrix}
0 & 0 & 0 & 0 \\ 0 & 0 & 0 & 0 \\ 0 & 0 & 0 & 0 \\ 0 & 0 & 0 & \frac{\bthe}{12\sigma_{s}}
\end{pmatrix}.
\]
Here we have used the equilibrium relation $\bet=\bthe^{4}$; see expressions \eqref{A0-b}-\eqref{Bjk-b} in Appendix \ref{non-gen-cou-md} for more details. From the expression above for $\bB$, we clearly have that
\begin{equation*}
    \text{ker}\left( \bB \right)= \text{span} \{ (1,0,0,0),(0,1,0,0),(0,0,1,0) \}.
\end{equation*}
Regarding $\bL$ and in view that $\bet = \bthe^{4}$, one can easily verify that the third and fourth rows of $\bL$ are linearly dependent, so that $\dim \ker \left( \bL \right)=3$. Since the vectors $(1,0,0,0)$, $(0,1,0,0)$ and $(0,0,1,4\bthe^{3})$ are linearly independent and belong to $\text{ker}\left( \bL \right)$, we have that 
\begin{equation*}
    \text{ker}\left( \bL \right) = \text{span}\{ (1,0,0,0),(0,1,0,0),(0,0,1,4\bthe^{3})\}.
\end{equation*}
Thus, there holds
\[
\text{ker}\left( \bL \right)\cap \text{ker}\left( \bB \right) = \text{span}\{ (1,0,0,0),(0,1,0,0) \}, 
\]
implying that $\psi \in \text{ker}\left( \bL \right) \cap \text{ker}\left( \bB \right)$, $\psi \neq 0 $, if and only if $\psi$ is of the form $\psi=(a_{1},a_{2},0,0)$ with $a_{1}$ and $a_{2}$ not simultaneously equal to zero. Hence, for vectors $\psi$   of this form  we have that
\[
\mu \bA^{0}\psi + \bA^{1}\psi = \begin{pmatrix}
\mu a_{1}\bthe\,\bet\,\bp_{\rho}/\brho + a_{1}\bu\,\bthe\, \bet\,\bp_{\rho}/\brho +a_{2}\bthe\,\bet\,\bp_{\rho} \\ \mu a_{2}\brho\,\bthe\,\bet + a_{1}\bthe\, \bet\,\bp_{\rho} +a_{2}\brho\, \bu\, \bthe\,\bet \\ a_{2}\bthe\, \bet\, \bp_{\theta} \\ a_{2}\bthe\, \bet/3 
\end{pmatrix},\quad \mu \in \R.
\]
Let us assume that for some $\mu \in \R$, $\mu \bA^{0}\psi + \bA^{1}\psi = 0$. As $\bthe$, $\bet$, and $\bp_{\theta}>0$, from the above computations we conclude that $a_{2}=0$. Now, as $\bp_{\rho}>0$, the second column in the expression for $\mu \bA^{0}\psi + \bA^{0}\psi$ implies that $a_{1}=0$, which contradicts the fact that $\psi \neq 0$. Thus, we have proved the following result.

\begin{lemma}
Under assumptions \hyperref[H1]{\rm{(H$_1$)}} -- \hyperref[H3]{\rm{(H$_3$)}}, the one dimensional linear system \eqref{1d-quas-symm} satisfies the genuine coupling condition. 
\end{lemma}

The lemma above and the Equivalence Theorem \ref{Eq-Th} imply the existence of a compensating matrix $K$ for system \eqref{1d-quas-symm}.
Then we have the following a priori energy estimate for the solutions of the initial value problem \eqref{neq-rad-hyd-1d}-\eqref{in-dat-1d}. 

\begin{lemma}\label{lem-apr-ee-1}
Let us consider the Cauchy problem \eqref{neq-rad-hyd-1d} and \eqref{in-dat-1d}, with $V_{0}(x)$ satisfying $V_{0}-\bV\in H^{s}(\R)^4$, $s\geq 2$. Let $V(x,t)$ be a local solution on a time interval $[0,T]$, then there exists a constant $a_{0}$ such that if
\[
N_{s}(T) \leq a_{0},
\]
then there holds the a apriori energy estimate
\begin{equation}
\label{apr-ee-1}
\Vert \partial_{x}V(\tau) \Vert_{s-1}^{2} + \int_{0}^{t} \Vert \partial_{x}^{2}\eta(\tau) \Vert_{s-1}^{2}+ \Vert P^{+}\partial_{x}V(\tau) \Vert_{s-1}^{2}\, d\tau \leq C_{1} \big( \Vert \partial_{x}V_{0} \Vert_{s-1}^{2}+ N_{s}(T)^{3} \big), 
\end{equation}
for every $t\in [0,T]$, and some positive constant $C_{1}=C_{1}(a_{0})$. Here $P^{+}$ denotes the orthogonal projection onto the range of $\bL$. 
\end{lemma}
\begin{proof}
See the proof of Lemma 3.1 in \cite{KaTh83}.
\end{proof}

\begin{lemma}
\label{lem-apr-ee-2}
Under the same assumptions as in Lemma \ref{lem-apr-ee-1} and for the same $a_{0}$, there exists a constant $C_{2}=C_{2}(a_{0})$ such that as long as 
\[
N_{s}(T) \leq a_{0}
\]
is satisfied, the following energy estimate 
\begin{equation}
\label{apr-ee-2}
\begin{aligned}
\int_{0}^{t} \Vert \partial_{x}^{2}(\rho, u, \theta)(\tau) \Vert_{s-2}^{2} \, d\tau - C_{2}\big( \Vert \partial_{x}V(t) \Vert_{s-1}^{2} &+\int_{0}^{t} \Vert \partial_{x}^{2} \eta (\tau) \Vert_{s-1}^{2} + \Vert P^{+}\partial_{x}V(\tau) \Vert_{s-1}^{2} \, d\tau \big) \\ &\leq C_{2}\big( \Vert \partial_{x}V_{0} \Vert_{s-1}^{2}+ N_{s}(T)^{3}  \big)
\end{aligned}
\end{equation}
holds for $t\in [0,T]$. As in Lemma \ref{lem-apr-ee-1}, $P^{+}$ is the orthogonal projection on the range of $\bL$. 
\end{lemma}
\begin{proof}
See the proof of Lemma 3.2 in \cite{KaTh83}.
\end{proof}
\begin{remark}
Lemmata \ref{lem-apr-ee-1} and \ref{lem-apr-ee-2} are essentially restatements of Lemmata 3.1 and 3.2 in \cite{KaTh83}, respectively. In this subsection we have simply verified that system \eqref{neq-rad-hyd-1d} satisfies the underlying hypotheses. 
\end{remark}

\begin{remark}
Notice that estimates \eqref{apr-ee-1} and \eqref{apr-ee-2} are not sufficient to get the global existence of solutions. For instance, if we multiply estimate \eqref{apr-ee-2} by $\alpha > 0$ and add it to estimate \eqref{apr-ee-1}, where $\alpha$ is such that $\alpha C_{2}<1$, we arrive at
\[
\begin{aligned}
\Vert \partial_{x}V(t) \Vert_{s-1}^{2}+& \int_{0}^{t}\Vert \partial_{x}^{2}(\rho, u, \theta)(\tau)\Vert_{s-2}^{2}+\Vert \partial_{x}^{2}\eta(\tau) \Vert_{s-1}^{2}\, d\tau \leq C \big( \Vert \partial_{x}V_{0} \Vert_{s-1}^{2}+N_{s}(T)^{3} \big),
\end{aligned}
\]
for some $C > 0$. However, in the left-hand side of the inequality above we are missing the terms $\Vert (V-\bV)(t) \Vert_{0}^{2}$ and $\int_{0}^{t}\Vert \partial_{x}V(\tau) \Vert_{0}^{2}\, d\tau$ to obtain an energy estimate as the one given by Theorem \ref{loc-ex-th}. 
\end{remark}

In the next section we perform a change of perturbation variables which will allow us to close the a priori energy estimates. The latter will eventually lead to the decay rate in time of solutions, which is a fundamental information not contained in the statements of Lemmata \ref{lem-apr-ee-1} and \ref{lem-apr-ee-2}. 

\section{Entropy and symmetrization}
\label{subsec:yong}

In this section we perform a change of variables to recast system \eqref{neq-rad-hyd}. The system in the new variables has some features that allow us to obtain a sharp linear energy estimate, which is needed to close the nonlinear analysis. This change of variables is motivated by the notion of an entropy function for a viscous system of conservation laws \cite{KaTh83,KaSh88a} and that of for a system of balance laws \cite{KY04,KY09}. These definitions are extensions of the classical notion of an entropy function for hyperbolic systems of conservation laws first introduced by Godunov \cite{Godu61b} and by Friedrichs and Lax \cite{FLa2}. 

\subsection{Entropies for non-conservative viscous balance systems}

Let us start by recalling the case of hyperbolic balance laws. Consider a system of the form
\begin{equation}\label{1}
W_{t} + f^{1}(W)_{x}=  Q(W).
\end{equation}
Here $x\in \R$ and $t\geq 0$ denote, respectively, the space and time variables. $W=(W_{1},\ldots, W_{n})$ represents the vector of quantities under examination  taking values in an open convex set $\mathcal{O}_{W}\subset \R^{n}$. The flux function $f^{1}$ and $Q$ are smooth functions of $W$ with values in $\R^{n}$. 

Let us define the sets
\[
\mathcal{M}:= \left\lbrace \psi \in \R^{n}: \: \< \psi, Q(W) \>=0, \;\, \text{for all} \; W\in \Or_{W} \right\rbrace
\]
and
\[
\Or_{\mathrm{eq}} = \left\lbrace W\in \Or_{W}: \: Q(W)=0 \right\rbrace.
\]
We have the following definitions.
\begin{definition}
\label{entr-rel-vis}
We say that system \eqref{1} has an entropy function if there exists a function $\cS:\Or_{W} \longrightarrow \R$ satisfying the following properties:
\begin{itemize}
\item[(a)]$\cS$ is a strictly convex function of $W$, that it, its Hessian $D_{W}^{2}\cS(W)$ is positive definite. 
\item[(b)]For $W\in \Or_{W}$ the matrix
\begin{equation}\label{ent-flux}
D_{W}f^{1}(W)D_{W}^{2}\cS (W)^{-1}
\end{equation}
is symmetric. 
\item[(c)]Let $W\in \Or_{W}$, then $W\in \Or_{\mathrm{eq}}$ if and only if $D_{W}\cS(W)^{\top}\in \mathcal{M}$. 
\item[(d)]For $W\in \Or_{\mathrm{eq}}$, the matrix $D_{W}Q(W)D_{W}^{2}\cS(W)^{-1}$ is symmetric and non-positive definite, and its null space coincides with $\mathcal{M}$. 
\end{itemize}
\end{definition}

One consequence of the existence of an entropy for system \eqref{1} is that we can write it in symmetric form. To this end, let $\cS$ be an entropy of system \eqref{1} in the sense just defined above. Therefore, the mapping
\begin{equation}\label{ent.}
U(W):= D_{W}\cS(W)^{\top}
\end{equation}
is a diffeomorphism from $\Or_{W}$ onto $\Or_{U}=U(\Or_{W})$ (this is a consequence of property (a)). Thus if we use $W=W(U)$ in \eqref{1} we obtain
\begin{equation}\label{2}
\tiA^{0}(U)U_{t}+ \tiA^{1}(U)U_{x}=  h(U),
\end{equation}
where
\begin{equation}\label{tiA0-A1}
\begin{aligned}
\tiA^{0}(U)&:= D_{U}W(U), \\
\tiA^{1}(U)&:=D_{U}f^{1}(U) = D_{W}f^{1}(W(U))D_{U}W(U), 
\end{aligned}
\end{equation}
and 
\[
h(U):= Q(W(U)). 
\]
Since $D_{U}W(U) = D_{W}^{2}\cS(W(U))^{-1}$ because of \eqref{ent.}, we have that
\[
\begin{aligned}
\tiA^{0}(U)&= D_{W}^{2}\cS(W(U))^{-1}, \\
\tiA^{1}(U)&= D_{W}f^{1}(W(U)) D_{W}^{2}\cS(W(U))^{-1},
\end{aligned}
\]
which are symmetric matrices with $\tiA^{0}>0$. 

The radiation hydrodynamics system under consideration is a set of viscous balance laws with non-conservative terms. Hence, let us consider a generic system of the form
\begin{equation}
\label{vis-bl-nc}
W_{t}+f^{1}(W)_{x}=(G(W)W_{x})_{x} + C(W)W_{x} + Q(W),
\end{equation}
where $G$ and $C$ are smooth function of $W$ with values in $\R^{n\times n}$. Consequently, we modify Definition \ref{entr-rel-vis} to account for non-conservative viscous balance systems having this form. In particular, we replace condition (b) by 
\begin{equation}\label{ii'}
\text{(b') For $W\in \Or_{W}$, the matrix $\left(D_{W}f^{1}(W)-C(W) \right)D_{W}^{2}\cS(W)^{-1}$ is symmetric},
\end{equation}
and add a new condition for the viscous term:
\begin{equation}
\text{(e) For $W\in \Or_{W}$, $G(W)D_{W}^{2}\cS(W)^{-1}$ is symmetric and positive semi-definite}.    
\end{equation}

\begin{definition}
\label{def-vis-bl-nc}
We say that system \eqref{vis-bl-nc} has an entropy function if there exists a function $\cS:\Or_{W} \longrightarrow \R$ satisfying the conditions (a), (b'), (c), (d) and (e) above.
\end{definition}

\subsection{Entropy for non-equilibrium radiation hydrodynamics}

Now, we shall verify that the entropy function proposed by Buet and Despr\'es \cite{BuDe04} ``almost'' satisfies Definition \ref{def-vis-bl-nc} for the one-dimensional version of system \eqref{neq-rad-hyd}. Indeed, regarding condition (b'), it will hold only for $W\in \Or_{\mathrm{eq}}$. Although this may seem as a difficulty, it turns out that it is sufficient for our purposes, as we shall see in the sequel. Let us start by writing the system under consideration as
\begin{equation}
\label{rad-hyd-1d}
\begin{aligned}
\partial_{t}\rho + \partial_{x}(\rho u)&=0, \\
\partial_{t}(\rho u) + \partial_{x} \big( \rho u^{2} + \big( p +\tfrac{1}{3}\eta\big) \big)&=0, \\ 
\partial_{t}\big( \rho E \big) + \partial_{x} \big( \rho E u + \big(p +\tfrac{1}{3}\eta \big)u \big)&= -\sigma_{a}\big( \theta^{4} - \eta \big)+\tfrac{1}{3}\eta u_{x}, \\
\partial_{t}\eta + \partial_{x}(\eta u )  &= \tfrac{1}{3\sigma_{s}}\eta_{xx} + \sigma_{a} (\theta^{4} -\eta)- \tfrac{1}{3}\eta  u_{x},
\end{aligned}
\end{equation}
which is the one dimensional version of system \eqref{neq-rad-hyd-var}. Here the variables describing the fluid are the mass density $\rho$, the scalar velocity $u$, the pressure $p$, the total energy $E$ and the absolute temperature $\theta$, while $\eta$ is the radiation intensity. The absorption coefficient $\sigma_{a}$ and the scattering coefficient $\sigma_{s}$ are, once again, positive constants.  

Clearly, this system is written in the form \eqref{vis-bl-nc} for the conserved quantities $W  : = (\rho,m,\cE,\eta)=(\rho, \rho u, \rho E, \eta)$ and with the notations
\begin{equation}\label{coef-W}
\begin{aligned}
f^{1}(W) & = \begin{pmatrix} \rho u \\ \rho u^{2} + p + \tfrac{1}{3}\eta \\ \rho E u + (p + \tfrac{1}{3}\eta)u \\ \eta u \end{pmatrix},\\
G(W) & = \begin{pmatrix}
0 & 0 & 0 & 0 \\ 0 & 0 & 0 & 0 \\ 0 & 0 & 0 & 0 \\ 0 & 0 & 0 & 1/(3\sigma_{s})
\end{pmatrix} =: \overline{G},\\
Q(W) & = \begin{pmatrix}
0 \\ 0 \\ \sigma_{a}\big( \eta - \theta^{4} \big) \\ -\sigma_{a}\big( \eta - \theta^{4} \big)\\
\end{pmatrix},\\
C(W)& = \begin{pmatrix}
0 & 0 & 0 & 0 \\ 0 & 0 & 0 & 0 \\ -u\eta/(3\rho) & \eta/(3\rho) & 0 & 0 \\ u\eta/(3\rho) & -\eta/(3\rho) & 0 & 0
\end{pmatrix}.
\end{aligned}
\end{equation}
Notice that $\overline{G}$ is a constant matrix. In the definitions above, we understand $u=m/\rho$, while $p$ and  $\theta$ are functions of $W$. This is indeed the case if we choose $\rho$ and $e$ as thermodynamic variables (see Remark \ref{thrm-rmrk}) so that we have
\[
p=p(\rho, e),\quad \theta=\theta(\rho,e ),
\]
and as 
\begin{equation}\label{eq:energy}
\cE= \rho E = \rho(e + u^{2}/2),
\end{equation}
we can write $e =e (\rho, m, \cE)$, implying that
\begin{equation}\label{p,the(W)}
p=p(\rho, m, \cE),\quad \theta=\theta(\rho, m , \cE).  
\end{equation}
The \emph{non conservative term} $C(W)W_x$ will play a crucial role in the sequel. It clearly comes from the differentiation
\begin{equation*}
u_{x}=\left( \frac{m}{\rho} \right )_{x} =-\frac{m}{\rho^{2}}\rho_x+ \frac{1}{\rho}m_{x} = -\frac{u}{\rho}\rho_{x} + \frac{1}{\rho}m_{x}.
\end{equation*}
Finally, the sets $\mathcal{M}$ and $\Or_{\mathrm{eq}}$ are given, respectively, by
\begin{equation}\label{Mcal}
\mathcal{M}= \mbox{span}\left\lbrace (1,0,0,0),(0,1,0,0),(0,0,1,1) \right\rbrace,
\end{equation}
and 
\begin{equation}\label{Eq-sts}
\Or_{\mathrm{eq}}=\left\lbrace W\in \Or_{W}:\: \eta= \theta^{4}(W) \right\rbrace. 
\end{equation}

As stated before, the variables we use to symmetrize the system are connected with the entropy functions of the model. In our case we consider the entropy function proposed by Buet and Despr\'es (see Corollary 2 in \cite{BuDe04}), namely
\begin{equation}
\label{form-Ent.}
\cS = -\rho s- \frac{4}{3}\eta^{3/4}.
\end{equation}
The quantity $\frac{4}{3}\eta^{3/4}$ is formally the radiative (physical) entropy at equilibrium. The function $s$ appearing in  \eqref{form-Ent.}  denotes the specific entropy (per unit mass) of the fluid and, as in the case of $p$ and $\theta$ in \eqref{p,the(W)}, it is understood as a function of $W$ which satisfies
\begin{equation}\label{therm-rel-2}
s_{\rho}=-p_{\theta}/\rho^{2},\qquad s_{\theta}=e_{\theta}/\theta. 
\end{equation}
It turns out that the function $\cS$ defined above works as a (mathematical) entropy for our system in the sense of Definition \ref{def-vis-bl-nc}, where, as already pointed out, condition (b') will hold only at equilibrium. We claim, for instance, that the needed symmetrizer is given by the Hessian of the entropy function, $D_{W}^{2}\cS(W)$. 
\begin{remark}\label{rem:varV}
Given that the pressure $p$, the internal energy $e$ and the 
 specific entropy $s$ are related through the thermodynamic relations \eqref{therm-rel} when $\rho$ and $\theta$ are chosen as independent thermodynamic variables, it is more convenient to perform our computations by introducing the (non conserved) variables $V : = (\rho, u, \theta, \eta)$. Then, the expression \eqref{eq:energy} defines a diffeomorphism $W=W(V)$ and the differentiation with respect to $W$ is performed using the chain rule involving  the matrix $\left . D_WV(W) = D_{V}W(V)^{-1}\right |_{V=V(W)}$, where
\[
D_{V}W(V)= \begin{pmatrix}
1 & 0 &0 & 0 \\ u & \rho & 0 & 0 \\ e+u^{2}/2 +\rho e_{\rho} & \rho u & \rho e_{\theta} & 0 \\ 0 & 0 & 0 & 1
\end{pmatrix},
\]
which is clearly non-singular. A direct computation then yields
\[
 D_{V}W(V)^{-1}= \begin{pmatrix}
 1 & 0 & 0 & 0 \\
 -u/\rho & 1/\rho &0 & 0 \\
  (\tfrac{1}{2}u^{2} -e-\rho\,e_{\rho})/(\rho e_{\theta}) & -u/(\rho e_\theta) & 1/(\rho e_\theta) & 0 \\
  0 & 0 & 0 & 1 \\
\end{pmatrix}.
\]
Finally, we obtain
\[
D_{V}\cS(W(V))= \big(-s-\rho s_{\rho}, 0, -\rho s_{\theta}, -1/\eta^{1/4} \big),
\]
and 
\[
D_{V}f^{1}(V)= \begin{pmatrix}
u & \rho & 0 & 0 \\ u^{2} + p_{\rho} & 2\rho u & p_{\theta} & 1/3 \\ u \big(e +\tfrac{1}{2}u^{2} \big)+ \rho u e_{\rho}+ p_{\rho}u & \rho\big( e +\tfrac{1}{2}u^{2}\big) + \rho u^{2} + p + \eta/3 & \rho u e_{\theta} + p_{\theta}u & u/3 \\ 0 & \eta & 0 & u
\end{pmatrix}.
\]
\end{remark}
In view of the above remark, by the chain rule we have that
\[
D_{W}\cS(W) = D_{V}\cS(V(W))D_{W}V(W) = D_{V}\cS(V(W))D_{V}W(V)^{-1}|_{V=V(W)}.
\]
Thus, after some straightforward computations and using relations \eqref{therm-rel} we obtain
\begin{equation}
\label{Z(V)}
\Theta(V):= U(W(V)) =D_{W}\cS(W(V))^{\top}= \begin{pmatrix}
-s+\big( e -u^{2}/2 +p/\rho \big)/\theta \\ u/\theta \\ -1/\theta \\ -1/\eta^{1/4}
\end{pmatrix},
\end{equation}
where $U$ is as in \eqref{ent.}. From this expression it is easy to see that for $W\in \Or_{W}$, $W\in \Or_{\mathrm{eq}}$ if and only if $D_{W}\cS(W)^{\top} \in \mathcal{M}$. This is true because, if $W\in \Or_{\mathrm{eq}}$, then $\eta= \theta^{4}$, 
which means that the third and fourth entries of $D_{W}\cS(W)^{\top}$ are the same. 

Next, we are going to show that $D_{W}^{2}\cS(W)$ is positive definite. Once again, we verify this fact indirectly, using that $W=W(V)$. The chain rule implies that
\begin{equation}
\label{*}
\begin{aligned}
D_{W}^{2}\cS(W) = D_{W}U(W) &= D_{V}\Theta(V(W))D_{W}V(W) \\ &= D_{V}\Theta(V(W))D_{V}W(V)^{-1} |_{V=V(W)}.
\end{aligned}
\end{equation}
Using relations \eqref{therm-rel} it is easy to see that
\[
D_{V}\Theta(V(W)) = \begin{pmatrix}
p_{\rho}/\rho \theta & -u/\theta & -(e-\tfrac{1}{2}u^{2} + \rho e_{\rho})/\theta^{2} & 0 \\ 0 & 1/\theta & -u/\theta^{2} & 0 \\ 0 & 0 & 1/\theta^{2} & 0 \\ 0 & 0 & 0 & 1/\big( 4\eta^{5/4} \big)
\end{pmatrix}.
\] 
Now, let us define the matrix 
\[
A(W):=D_{V}W(V)^{\top}|_{V=V(W)} D_{W}^{2}\cS(W) D_{V}W(V)|_{V=V(W)}. 
\]
Thus using \eqref{*} and after some straightforward computations we get
\[
A(W)= D_{V}W(V)^{\top}|_{V=V(W)} D_{V}\Theta(V(W)) = \begin{pmatrix}
p_{\rho}/\theta \rho & 0 & 0 & 0 \\ 0 & \rho/\theta & 0 & 0 \\ 0 & 0 & \rho e_{\theta}/\theta^{2} & 0 \\ 0 & 0 & 0 & 1/\big( 4\eta^{5/4}\big) 
\end{pmatrix},
\]
which is positive definite, and so is the matrix $D_{W}^{2}\cS(W)$. By \eqref{*} we obtain
\begin{equation}\label{coeffA0}
\begin{aligned}
\tiA^{0}(U) &= D_{W}^{2}\cS(W)^{-1}|_{W=W(U)} \\ &= \left( D_{V}W(V)|_{V=V(W)} D_{V}\Theta(V(W))^{-1} \right) |_{W=W(U)} \\ &= 
\begin{pmatrix}
\rho\, \theta / p_{\rho} & \rho\, u\, \theta / p_{\rho}  & a & 0 \\ \rho \, u\, \theta / p_{\rho} & \rho\, \theta  (1+u^{2}/p_{\rho}) & b & 0 \\ a & b & c & 0 \\ 0 & 0 & 0 & 4\eta^{5/4}
\end{pmatrix},
\end{aligned}
\end{equation}
where $U$ is defined by \eqref{ent.}, and the coefficients $a$, $b$ and $c$ are given by
\[
\begin{aligned}
a &=\frac{\theta \big( e_{\rho}\rho^{2} + \rho(e + u^{2}/2) \big)}{p_{\rho}},\\
b &=\frac{u\, \theta \big( e_{\rho}\rho^{2} + \rho (e + u^{2}/2) + \rho\, p_{\rho}  \big)}{p_{\rho}},\\
c &= \frac{\rho\, \theta \big( \big( (e+ u^{2}/2) + e_{\rho}\rho \big)^{2} +  (u^{2}+\theta\, e_{\theta})p_{\rho} \big)}{ p_{\rho} }.
\end{aligned}
\]
Now, use the chain rule to write
\[
D_{W}f^{1}(W)=D_{V}f^{1}(V(W)) D_{W}V(W) = D_{V}f^{1}(V(W))D_{V}W(V)^{-1} |_{V=V(W)},
\]
so that 
\[
D_{W}f^{1}(W)D_{W}^{2}\cS(W)^{-1} = D_{V}f^{1}(V(W))D_{V}\Theta(V(W))^{-1}.
\]
Thus
\begin{equation}\label{redef-A1}
\begin{aligned}
\tiA^{1}(U)&:=\left(D_{W}f^{1}(W)-C(W)  \right)D_{W}^{2}\cS(W)^{-1} |_{W=W(U)} \\ 
& =\left( D_{V}f^{1}(V(W))D_{V}\Theta(V(W))^{-1} - C(W)D_{V}W(V)|_{V=V(W)}D_{V}\Theta(V(W))^{-1}\right) |_{W=W(U)} \\
&= \left[ D_{V}f^{1}(V(W))-  C(W)D_{V}W(V)|_{V=V(W)} \right] D_{V}\Theta(V(W))^{-1} |_{W=W(U)},
\end{aligned}
\end{equation}
and after some computations we get
\begin{equation}\label{coeffA1}
\tiA^{1}(U)= \begin{pmatrix}
a_{11} & a_{12} & a_{13} & 0 \\ a_{12} & a_{22} & a_{23} & a_{24} \\ a_{13} & a_{23} & a_{33} & a_{34} \\ 0 & a_{42} & a_{43} & a_{44}
\end{pmatrix}
\end{equation}
where the coefficients are given by
\[
\begin{aligned}
a_{11} &=\rho\, u\, \theta/p_{\rho},\\
a_{12} &= \rho\, \theta\big( p_{\rho}  + u^{2} \big)/ p_{\rho},\\
a_{13} &= u \, \theta\big( e_{\rho}\rho^{2} + \rho \big( e + u^{2}/2 \big) + \rho \, p_{\rho} \big)/p_{\rho},\\
a_{22} &= \rho\, u \, \theta \big(3 p_{\rho} + u^{2}\big)/p_{\rho},\\
a_{23} &=\theta\big( \rho\,u^{2}\big(e+ u^{2}/2 \big)+ u^{2}e_{\rho}\rho^{2} + \big(  p + \rho\, e + 5\rho\, u^{2}/2 \big)p_{\rho} \big)/p_{\rho},\\
a_{24} &=4\eta^{5/4}/3,\\
a_{33} &= u\, \theta \big( ( \rho^{2}e_{\rho}+\rho(e + u^{2}/2))^{2}+\rho (2(p+(e+u^{2})\rho)+\rho\, \theta\, e_{\theta})p_{\rho} \big)/(\rho\, p_{\rho}), \\
a_{34} &=4u\,\eta^{5/4}/3,\\
a_{42} &=4\theta\, \eta/3,\\
a_{43} &=4u\, \theta \, \eta/3, \\
a_{44} &=4 u \, \eta^{5/4}. 
\end{aligned}
\]
From the expression for the coefficients $a_{24}$, $a_{34}$, $a_{42}$, and $a_{43}$ is easy to see that $A^{1}(U)$ is symmetric as long as $\eta = \theta^{4}$, that is, for $W\in \Or_{\mathrm{eq}}$. This shows that condition (b') is satisfied for $W\in \Or_{\mathrm{eq}}$. Observe that in \eqref{redef-A1} we have redefined the expression for $\tiA^{1}(U)$ in \eqref{tiA0-A1} in order to account for the non-conservative term that appears in \eqref{vis-bl-nc}. To compute $D_{W}Q(W)D_{W}^{2}\cS(W)^{-1}$ we proceed in the same fashion as we did for the term $D_{W}f^{1}(W)D_{W}^{2}\cS(W)^{-1}$. This yields
\[
\begin{aligned}
D_{W}Q(W)D_{W}^{2}\cS(W)^{-1} &= D_{V}Q(V(W))D_{V}\Theta(V(W))^{-1} \\
&= \begin{pmatrix}
0 & 0 & 0 & 0 \\ 0 & 0 & 0 & 0 \\ 0 & 0 & -4\sigma_{a}\theta^{5} & 4\sigma_{a}\eta^{5/4} \\ 0 & 0 & 4\sigma_{a}\theta^{5} & -4\sigma_{a}\eta^{5/4} 
\end{pmatrix},
\end{aligned}
\]
which is non-positive definite for $W\in \Or_{\mathrm{eq}}$ and its null space coincides with $\mathcal{M}$ given by \eqref{Mcal} for $W\in \Or_{\mathrm{eq}}$. 

Finally, let us compute $G(W)D_{W}^{2}\cS(W)$. For that purpose we use the relation for $D_{W}^{2}\cS(W)$ given by \eqref{*} and obtain
\[
\begin{aligned}
G(W)D_{W}^{2}\cS(W)^{-1} &= \overline{G}D_{V}W(V)|_{V=V(W)} D_{V}\Theta(V(W))^{-1}  \\
&= \begin{pmatrix}
0 & 0 & 0 & 0 \\ 0 & 0 & 0 & 0 \\  0 & 0 & 0 & 0 \\ 0 & 0 & 0 & 4\eta^{5/4}/(3\sigma_{s})
\end{pmatrix},
\end{aligned}
\]
which is clearly positive semi-definite. Thus, we have shown that $\cS$ given by \eqref{form-Ent.} is indeed an entropy function according to the modified Definition \ref{def-vis-bl-nc} for system \eqref{rad-hyd-1d}. 

\subsection{New perturbation variables}
\label{secnpv}
Next, let us consider a constant state $\bW \in \Or_{\mathrm{eq}}$ and define
\begin{equation}
\label{def-z}
Z := D_{U}W(\bU)^{-1}(W-\bW),
\end{equation}
where $\bU = U(\bW)$. Thus $W-\bW = D_{U}W(\bU)Z$. Now observe that
\[
\begin{aligned}
f^{1}(W)_{x} &= D_{W}f^{1}(\bW)(W-\bW)_{x} + \left(f^{1}(W)-f^{1}(\bW)-D_{W}f^{1}(\bW)(W-\bW) \right)_{x} \\ &=  D_{W}f^{1}(\bW)D_{U}W(\bU)Z_{x} + \left(f^{1}(W)-f^{1}(\bW)-D_{W}f^{1}(\bW)(W-\bW) \right)_{x},\\
\big(G(W)W_{x}\big)_{x} &= \overline{G} W_{xx} = \overline{G} D_{U}W(\bU)Z_{xx},\\
C(W)W_{x}&= C(\bW)W_{x} + \big( C(W) - C(\bW)\big) W_{x}  \\ &= C(\bW)D_{U}W(\bU)Z_{x} + \big( C(W) - C(\bW)\big) W_{x},
\end{aligned}
\]
and
\[
\begin{aligned}
Q(W)& = D_{W}Q(\bW)(W-\bW) + \left( Q(W)-Q(\bW)-D_{W}Q(\bW)(W-\bW) \right) \\ &= D_{W}Q(\bW)D_{U}W(\bW)Z + \left( Q(W)-Q(\bW)-D_{W}Q(\bW)(W-\bW) \right). 
\end{aligned}
\]
Then system \eqref{rad-hyd-1d} transforms into 
\begin{equation}\label{sys-z}
\tiA^{0}Z_{t} + \tiA^{1}Z_{x}+ \tiL Z = \tiB Z_{xx} + g_{x}+ q,
\end{equation}
where
\[
\begin{aligned}
\tiA^{0}&=\tiA^{0}(\bU), \\  
\tiA^{1}&=\tiA^{1}(\bU), \\  
\tiL :&=-D_{W}Q(\bW)D_{W}^{2}\cS(\bW)^{-1},  \\ \tiB:&=G(\bW)D_{W}^{2}\cS(\bW)^{-1},
\end{aligned}
\]
with $\tiL$ and $\tiB$ positive semi-definite, while $g$ and $q$ are given by
\begin{equation}\label{pq-form}
\begin{aligned}
g&= - \left(f^{1}(W)-f^{1}(\bW)-D_{W}f^{1}(\bW)(W-\bW) \right), \\ 
q&= \big( C(W) - C(\bW)\big) W_{x} +  \left( Q(W)-Q(\bW)-D_{W}Q(\bW)(W-\bW) \right). 
\end{aligned}
\end{equation}
We can rewrite $g$ and $q$ above by noticing that
\[
\big( C(W)- C(\bW) \big)W_{x} = \begin{pmatrix}
0 \\ 0 \\ \frac{1}{3}(\eta - \bet )u_{x} \\ -\frac{1}{3}(\eta - \bet )u_{x} \end{pmatrix} = \begin{pmatrix}
0 \\ 0 \\ \frac{1}{3}(\eta - \bet )(u - \bu) \\ -\frac{1}{3}(\eta - \bet )(u - \bu)
\end{pmatrix}_{x} + \begin{pmatrix}
0 \\ 0 \\ -\frac{1}{3}(u-\bu)\eta_{x}  \\ \frac{1}{3}(u-\bu)\eta_{x}
\end{pmatrix},
\]
so that 
\[
g_{x} + q = \widetilde{g}_{x}+ \widetilde{q},
\]
where
\begin{equation}
\label{tipq-form}
\begin{aligned}
\widetilde{g} &= g + \begin{pmatrix} 0 \\ 0 \\ \frac{1}{3}(\eta -\bet)(u -\bu) \\ -\tfrac{1}{3}(\eta -\bet)(u -\bu) \end{pmatrix}, \\
\widetilde{q}&= \begin{pmatrix}
0 \\ 0 \\ -\frac{1}{3}(u -\bu) \eta_{x} \\ \frac{1}{3}(u- \bu)\eta_{x}
\end{pmatrix} + \left( Q(W)-Q(\bW)-D_{W}Q(\bW)(W-\bW) \right).
\end{aligned}
\end{equation}
Essentially, in order to obtain system \eqref{sys-z} one multiplies system \eqref{vis-bl-nc} on the left by the Hessian of the entropy function evaluated at $\bU$. 

\begin{remark}\label{rem:nonconsort}It is worth observing that the special shape of the non conservative terms implies in particular that the full quantity  $\widetilde{q}$ belongs to $\mathcal{M}^{\bot}$, where  $\mathcal{M}$ is the orthogonal space  of $Q(W)$; see \eqref{Mcal}. This feature will be crucial in the sequel, to obtain the improved linear decay rate established by Lemma \ref{lin-dec-shr-lem}  below.
\end{remark}

%

To sum up, we collect the previous observations into the following result.
\begin{proposition}
\label{propentropy}
The Hessian of the entropy function $\cS$ defined in \eqref{form-Ent.} is a symmetrizer for the linearization of system \eqref{rad-hyd-1d} around a constant equilibrium state and through the change of variables defined in \eqref{def-z}.
\end{proposition}


To finish this section we state a result (without proof) that relates the solutions to the Cauchy problem \eqref{neq-rad-hyd-1d} and \eqref{in-dat-1d} with the ones for system \eqref{sys-z} with initial data $Z_{0}$ corresponding to the initial data $V_{0}$ for \eqref{neq-rad-hyd-1d} through the change of variables \eqref{def-z}. The proof of the result is based on the following observation. We have already mentioned that the correspondence $V \mapsto W$ defines a diffeomorphism; see Remark \ref{rem:varV} above. Thus for $V$ close to $\bV $ we can write
\[
W - \bW = D_{V}W(\bV_{*})(V-\bV),
\]
which, in turn, implies that
\[
Z= D_{U}W(\bU)(W-\bW)=D_{U}W(\bU)D_{V}W(\bV_{*})(W-\bW),
\]
for some $\bV_{*}$ between $V$ and $\bV$. Then, if the solutions $V$ remain close to $\bV$ (see Remark \ref{cld-loc-sol-rm}), so do the variables $W$ and $\bW$, and from the expressions above we get that Sobolev norms for any two of the three perturbed variables $V-\bV$, $W-\bW$ and $Z$ are equivalent. For more details, see the proof of Lemma 3.7 in \cite{PlV2}.

\begin{proposition}\label{Eq-norm-prop}
Assume \hyperref[H1]{\rm{(H$_1$)}} -- \hyperref[H3]{\rm{(H$_3$)}} and consider the initial value problem of system \eqref{neq-rad-hyd-1d} with initial data $V_{0}(x)$ such that $V_{0}-\bV\in H^{s}(\R)^{4}$, $s\geq 3$. Let $V(x,t)$ be the local solution in the interval $[0,T]$, $T>0$, which is given by Theorem \ref{loc-ex-th-1d}. Then the new perturbation variables
\[
Z(x,t)=D_{U}W(\bU)(W(V(x,t))-W(\bV))
\]
solve system \eqref{sys-z} with initial data
\[
Z_{0}(x)=Z(x,0)=D_{U}W(\bU)(W(V_{0}(x))-W(\bV)).
\]
In addition, the following relations hold:
\begin{itemize}
\item[(i)] There exist positive constants $c_{0}$, $C_{0}$ such that
\[
c_{0}\Vert (V-\bV)(t) \Vert_{k} \leq \Vert Z(t) \Vert_{k} \leq C_{0} \Vert (V-\bV)(t) \Vert_{k},
\]
for all $0\leq k \leq s$, and $t\in [0,T]$.
\item[(ii)] 
There exist constants $c_{1}$, $C_{1}>0$ such that 
\[
c_{1}\Vert (V-\bV)(t) \Vert_{k} \leq \Vert (W-\bW)(t) \Vert_{k} \leq C_{1} \Vert (V-\bV)(t) \Vert_{k},
\]
for all $0 \leq k \leq s$, and $t\in [0,T]$.
\item[(iii)] If, in addition, the initial perturbation satisfies $V_{0}-\bV\in L^{1}(\R)^4$, so does the initial data for the system in $Z$, and there holds
\[
\Vert Z_{0} \Vert_{L^{1}} \leq C \Vert V_{0}-\bV \Vert_{L^{1}},
\]
for some constant $C>0$.
\end{itemize}
\end{proposition}
\section{Dissipative structure and linear decay rate of solutions}
\label{sec:linear}

As we have already mentioned at the end of Section \ref{sec:local}, we need an additional energy estimate to get the correct a priori estimate and to complete the nonlinear analysis. To this end, in this section we examine the dissipative structure of the system at the linear level, which  will also imply the linear decay rate of solutions. These linear properties play a key role to obtain the desired a priori estimate and the decay rate of the solutions for the full nonlinear system.

It is to be noticed that we could have performed the analysis in the system written in the variables $V$. Indeed, the computations prior to the statement of Lemmata \ref{lem-apr-ee-1} and \ref{lem-apr-ee-2} imply the dissipative structure of the system. However, we need a sharper linear energy estimate to close the nonlinear one and this is possible only if we work with the set of variables identified in the previous Section \ref{secnpv}.

\subsection{Dissipative structure of the linear system}
Let us start by recalling the system that we obtained in Section \ref{subsec:yong}. Let us fix a constant state $\bV$. To this constant state corresponds a (unique) constant state in the conserved variables $W=(\rho, \rho u, \rho E, \eta)$ which we denote by $\bW$;  see Section \ref{subsec:yong}, Remark \ref{rem:varV}. Thus the system in the new perturbed variables reads
\begin{equation}\label{sys-z-i}
\tiA^{0}Z_{t} + \tiA^{1}Z_{x}+\tiL Z = \tiB Z_{xx} + \tilde{g}_{x} + \tilde{q},
\end{equation}
where
\[
Z = D_{U}W(\bU)( W -\bW),
\]
with $U$ being the gradient with respect to $W$ of the entropy function of the fluid, and $\bU=U(\bW)$. The (constant) matrix coefficients are given by
\begin{equation}\label{btiA0-A1}
\begin{aligned}
\tiA^{0}&=\tiA^{0}(\bU)= \begin{pmatrix}
\brho\, \bthe / \bp_{\rho} & \brho\, \bu \, \bthe / \bp_{\rho}  & \bar{a} & 0 \\ \brho \, \bu \, \bthe / \bp_{\rho} & \brho\, \bthe  (1+\bu^{2}/\bp_{\rho}) & \bar{b} & 0 \\ \bar{a} & \bar{b} & \bar{c} & 0 \\ 0 & 0 & 0 & 4\bet^{5/4}
\end{pmatrix}, \\
\tiA^{1}&=\tiA^{1}(\bU) = \begin{pmatrix}
\bar{a}_{11} & \bar{a}_{12} & \bar{a}_{13} & 0 \\ \bar{a}_{12} & \bar{a}_{22} & \bar{a}_{23} & \bar{a}_{24} \\ \bar{a}_{13} & \bar{a}_{23} & \bar{a}_{33} & \bar{a}_{34} \\ 0 & \bar{a}_{24} & \bar{a}_{34} & \bar{a}_{44}
\end{pmatrix},
\end{aligned}
\end{equation}
already computed in the previous section and evaluated here at $\bU=U(W(\bV))$; see in particular
\eqref{coeffA0} and \eqref{coeffA1} and the subsequent expressions. Moreover, 
\begin{equation}\label{tiL-B}
\begin{aligned}
\tiL &= \begin{pmatrix}
0 & 0 & 0 & 0 \\ 0 & 0 & 0 & 0 \\ 0 & 0 & 4\sigma_{a} \bthe^{5} & -4\sigma_{a}\bet^{5/4} \\ 0 & 0 & -4\sigma_{a}\bthe^{5} & 4\sigma_{a}\bet^{5/4} 
\end{pmatrix}, \\
\tiB &= \begin{pmatrix}
0 & 0 & 0 & 0 \\ 0 & 0 & 0 & 0 \\  0 & 0 & 0 & 0 \\ 0 & 0 & 0 & 4\bet^{5/4}/3\sigma_{s}
\end{pmatrix}
\end{aligned}
\end{equation}
and the terms $\tilde{g}$ and $\tilde{q}$,   defined in \eqref{tipq-form}, are such that
\begin{equation}\label{order-pq}
\begin{aligned}
\tilde{g} &= O \left( \vert W - \bW \vert^{2} \right), \\
\tilde{q} &= O \left( \vert W - \bW \vert \vert \eta_{x} \vert + \vert W- \bW \vert^{2} \right),
\end{aligned}
\end{equation}
with $\tilde{q}\in \mathcal{M}^{\bot}$, where $\mathcal{M}$ is the null space of $\tiL$. 

Next we verify that system \eqref{sys-z-i} at the linear level, that is, 
\begin{equation}\label{sys-z-lin}
\tiA^{0}Z_{t}+ \tiA^{1}Z_{x}+\tiL Z = \tiB Z_{xx},
\end{equation}
satisfies the genuine coupling condition. To this end, observe that, from the expression for $\tiL$ and $\tiB$, it is easy to verify that
\[
\ker \tiL \cap \ker \tiB= \mbox{span}\left\lbrace (1,0,0,0),(0,1,0,0) \right\rbrace.
\]
Then if $\psi \in \ker \tiL \cap \ker \tiB$, $\psi \neq 0$, then $\psi$ is of the form $\psi = (a_{1},a_{2},0,0)$ with $a_{1}$ and $a_{2}$ not being simultaneously zero. Thus we have 
\begin{equation}\label{gen-cou-ver}
\mu \tiA^{0}\psi + \tiA^{1}\psi = \begin{pmatrix}
\mu a_{1}\brho\, \bthe /\bp_{\rho} + \mu a_{2}\brho\, \bu \, \bthe/ \bp_{\rho} + a_{1}\bar{a}_{11}+a_{2}\bar{a}_{12} \\ \mu a_{1}\brho\, \bu \, \bthe/ \bp_{\rho} +\mu a_{2}(1+\bu^{2}/ \bp_{\rho}) + a_{1}\bar{a}_{12} + a_{2}\bar{a}_{22} \\ \mu a_{1}\bar{a} + \mu a_{2}\bar{b} + a_{1}\bar{a}_{13} + a_{2}\bar{a}_{23} \\ a_{2}\bar{a}_{24}
\end{pmatrix}.
\end{equation}
Let us assume that $\mu \tiA^{0}\psi + \tiA^{1}\psi = 0$, for $\psi =(a_{1},a_{2}, 0, 0) \neq 0$. Then, as $\bar{a}_{24}=4\bthe^{5}/3$, we conclude that $a_{2}=0$. Using this, the first row of \eqref{gen-cou-ver} implies that
\[
0= \mu a_{1}\brho\, \bthe/\bp_{\rho} + a_{1}\bar{a}_{11}= \mu a_{1}\brho\, \bthe/\bp_{\rho} +a_{1}\brho \, \bu \, \bthe/\bp_{\rho},
\]
so that $\mu = -\bu$. However, if we use this relation for the expression in the second row of \eqref{gen-cou-ver} we obtain
\[
0= \mu a_{1}\brho \, \bu\, \bthe /\bp_{\rho} +a_{1}\bar{a}_{12} = -a_{1}\brho \, \bu^{2}\, \bthe / \bp_{\rho} + a_{1}\brho\, \bthe (1 +\bu^{2}/\bp_{\rho}) = a_{1}\brho\, \bthe,
\]
which holds only if $a_{1}=0$ as $\brho$, $\bthe>0$ and this contradicts the fact that $\psi \neq 0$. Therefore we have proved that for any $\psi \in \ker\tiL \cap \ker \tiB$, $\psi \neq 0$, it holds
\[
\mu \tiA^{0}\psi + \tiA^{1}\psi \neq 0,\quad \forall \mu \in \R.
\]
 We summarize the computations above in the following result. 

\begin{proposition}
The linear system \eqref{sys-z-lin} satisfies the genuine coupling condition. 
\end{proposition}
As a consequence of the Equivalence Theorem, there exists a compensating matrix for system \eqref{sys-z-lin} which can be used, in turn, to obtain the following pointwise energy estimate in the Fourier space for the solution of system \eqref{sys-z-lin}. This is a standard result that can be found in \cite{KaTh83,AnMP20}.

\begin{lemma}
The solutions $Z(x,t)$ of the linear system \eqref{sys-z-lin} satisfy the estimate
\begin{equation}\label{pw-ee-Four}
\vert \hZ(\xi, t) \vert \leq C \exp\left( -\frac{k \xi^{2}}{1+\xi^{2}}t \right) \vert \hZ(\xi, 0) \vert,
\end{equation}
for all $\xi \in \R$ and $t\geq 0$, with some positive constants $C$ and $k$. Here $\hZ$ denotes the Fourier transform of $Z$.
\end{lemma}
\begin{proof}
See the proof of Lemma 3.A.1 in \cite{KaTh83}, or that of Lemma 5.1 in \cite{AnMP20}. 
\end{proof}
\subsection{Linear decay of solutions}
In this section we obtain the decay rate of solutions to the linear system \eqref{sys-z-lin}. This is a direct consequence of the pointwise energy estimate \eqref{pw-ee-Four} in the Fourier space upon application of Plancherel's Theorem. For instance, estimate \eqref{pw-ee-Four} directly yields the following result, which establishes a linear decay rate for the solutions. We omit its proof because it can be easily obtained using the arguments of Lemma \ref{lin-dec-shr-lem} below, which contains an improved decay rate.

\begin{lemma}
Let us consider the Cauchy problem for system \eqref{sys-z-lin} with initial data $Z_{0}(x)\in \big(H^{s}(\R) \cap L^{1}(\R)\big)^4$, for $s\geq 0$. Then for each fixed $0 \leq \ell \leq s$ there holds
\begin{equation}
\label{lin-ene-est}
\Vert \partial_{x}^{\ell}Z(t) \Vert_{0} \leq C e^{-c_{1}t}\Vert \partial_{x}^{\ell}Z_{0} \Vert_{0} + C\big(1+t \big)^{-(\ell/2 + 1/4)} \Vert Z_{0} \Vert_{L^{1}}, 
\end{equation}
for $t\geq 0$, and some positive constants $C$, $c_{1}$.  
\end{lemma}

Now we define 
\begin{equation}\label{def-sem}
\left(e^{t\Phi}h \right)(x) := \frac{1}{(2 \pi)^{1/2}}\int_{\R}e^{t\Phi(i \xi)}\hat{h}(\xi)e^{i\xi x}\: d\xi,
\end{equation}
where
\[
\Phi(i \xi):= -\big( \tiA^{0} \big)^{-1}\big(\tiL + i\xi \tiA^{1} - (i \xi)^{2}\tiB \big).
\]
Then $\left( e^{t\Phi}h \right)(x)$ is the solution to the Cauchy problem of \eqref{sys-z-lin} with initial condition $Z_{0}=h$. Thus the estimate \eqref{lin-ene-est} can be rewritten as follows.
\begin{corollary}
\label{lin-dec-cor}
For $h \in \big(H^{s}(\R) \cap L^{1}(\R)\big)^4$, $s\geq 0$, we have the following estimate
\begin{equation}\label{lin-ene-est-sem}
\Vert \partial_{x}^{\ell}e^{t \Phi}h \Vert_{0} \leq C e^{-c_{1}t}\Vert \partial_{x}^{\ell} h \Vert_{0} + C\big(1+t \big)^{-(\ell/2 + 1/4)} \Vert h \Vert_{L^{1}}.
\end{equation}
for each (fixed) $0 \leq \ell \leq s$, and some positive constant $C$, $c_{1}$.
\end{corollary}

Next, we follow \cite{KY09} and improve estimate \eqref{lin-ene-est} in the case when $h \in \mathcal{M}^{\bot}$, where $\mathcal{M}$ is the null space of $\tiL$.
\begin{lemma}
\label{lin-dec-shr-lem}
Let $h \in \big(H^{s}(\R) \cap L^{1}(\R)\big)^4$, $s\geq 0$, be such that $h \in \mathcal{M}^{\bot}$. Then there holds
\begin{equation}
\label{lin-est-sem-imp}
\Vert \partial_{x}^{\ell}e^{t\Phi}\big( \tiA^{0} \big)^{-1} h \Vert_{0} \leq Ce^{-c_{1}t}\Vert \partial_{x}^{\ell}h \Vert_{0} + C(1+t)^{-(\ell/2 + 3/4)}\Vert h \Vert_{L^{1}},
\end{equation}
for all $0 \leq \ell \leq s$ and $t\geq 0$, and some positive constants $C$, $c_{1}$.
\end{lemma}
\begin{proof}
As $\Big( e^{t\Phi}\big( \tiA^{0} \big)^{-1} h \Big)(x)$ is the solution of the linear system \eqref{sys-z-lin} with initial data $\big(\tiA^{0}\big)^{-1}h(x)$, taking Fourier transform and using the pointwise energy estimate \eqref{pw-ee-Four}, we obtain
\begin{equation}\label{est-sem-hf}
\begin{aligned}
\vert e^{t\Phi(i\xi)}\big( \tiA^0 \big)^{-1}\hh(\xi) \vert^{2} &\leq C \exp\left(-\frac{2k\xi^2}{1+\xi^2}t \right) \vert \big( \tiA^0 \big)^{-1} \hh(\xi) \vert^2 \\ &\leq C \exp\left(-\frac{2k\xi^2}{1+\xi^2}t \right) \vert \hh(\xi) \vert^2,
\end{aligned}
\end{equation} 
which holds for all $\xi \in \R$. Now, we invoke Lemma \ref{sptr-anal-lem} in Appendix \ref{sptr-anal}: there exist positive constants $c$, $C$ and $R$ such that
\begin{equation}\label{est-sem-lf}
\vert e^{t\Phi(i\xi)}\big( \tiA^0 \big)^{-1}\hh(\xi) \vert^{2}\leq Ce^{-ct}\vert \hh(\xi) \vert^2 + C\vert \xi \vert^{2}e^{-c\xi^2 t}\vert \hh(\xi) \vert^{2},
\end{equation}
for $\vert \xi \vert \leq R$. Next, we choose and fix an integer $\ell$ such that $0\leq \ell \leq s$, multiply $\vert e^{t\Phi(i\xi)}\big( \tiA^0 \big)^{-1}\hh(\xi) \vert^{2}$ by $\vert \xi \vert^{2\ell}$ and integrate over all $\xi\in \R$ to get
\begin{equation}\label{est-imp}
\begin{aligned}
\int_{\R}\vert \xi \vert^{2\ell}\vert e^{t\Phi(i\xi)}&\big( \tiA^0 \big)^{-1}\hh(\xi) \vert^{2}\, d\xi =  
\int_{\vert \xi \vert \leq R}\vert \xi \vert^{2\ell}\vert e^{t\Phi(i\xi)}\big( \tiA^0 \big)^{-1}\hh(\xi) \vert^{2}\, d\xi + \\ & + \int_{\vert \xi \vert \geq R}\vert \xi \vert^{2\ell}\vert e^{t\Phi(i\xi)}\big( \tiA^0 \big)^{-1}\hh(\xi) \vert^{2}\, d\xi =: J_1 + J_2,
\end{aligned} 
\end{equation} 
with the same $R$ as the one for which estimate \eqref{est-sem-lf} holds.

We estimate $J_2$ first. In this case, for $\vert \xi \vert \geq R$ we have $\exp \big(-\frac{2k\xi^2}{1+\xi^2}t \big)\leq e^{-c_{1}t}$, for some constant constant $c_{1}>0$. This implies that 
\begin{equation}
\label{est-imp-i}
\begin{aligned}
\int_{\vert \xi \vert \geq R}\vert \xi \vert^{2\ell}\vert e^{t\Phi(i\xi)}\big( \tiA^0 \big)^{-1}\hh(\xi) \vert^{2}\, d\xi  
&\leq Ce^{-c_{1}t} \int_{\vert \xi \vert \geq R} \vert \xi \vert^{2\ell}\vert \hh(\xi) \vert^{2}\, d\xi \\ & \leq Ce^{-c_{1}t}\Vert \partial_{x}^{\ell}h \Vert^2,
\end{aligned}
\end{equation}
where we have used estimate \eqref{est-sem-hf}. In the case of $J_1$, we use estimate \eqref{est-sem-lf} to get
\begin{equation}\label{est-imp-ii}
\begin{aligned}
\int_{\vert \xi \vert \leq R}\vert \xi \vert^{2\ell}\vert e^{t\Phi(i\xi)}\big( \tiA^0 \big)^{-1}\hh(\xi) \vert^{2}\, d\xi &\leq  \int_{\vert \xi \vert \leq R} Ce^{-ct}\vert \xi \vert^{2\ell} \vert \hh(\xi) \vert^2\, d\xi  \\ &\ + \int_{\vert \xi \vert \leq R} C\vert \xi \vert^{2(\ell+1)}e^{-c\xi^{2}t}\vert \hh(\xi) \vert^2\, d\xi \\ &\leq Ce^{-ct}\Vert \partial_{x}^{\ell}h \Vert^{2} + C(1+t)^{-(3/2+\ell)}\Vert h \Vert_{L^1}^{2}.
\end{aligned}
\end{equation}
Thus combining \eqref{est-imp}, \eqref{est-imp-i} and \eqref{est-imp-ii} together with the Plancherel's Theorem we get
\[
\Vert \partial_{x}^{\ell}e^{t\Phi}\big(\tiA^{0}\big)^{-1} h \Vert^{2}_{0} \leq Ce^{-ct}\Vert \partial_{x}^{\ell} h \Vert^{2}_{0} + C(1+t)^{-(3/2+\ell)}\Vert h  \Vert_{L^1}^{2},
\]
which implies estimate \eqref{lin-est-sem-imp} after taking the square root. 
\end{proof}
This last estimate is going to play an important role in the establishment of the nonlinear energy estimate on the system \eqref{sys-z}, where we use the fact that $\widetilde{q} \in \mathcal{M}^{\bot}$; see \eqref{tipq-form} and the subsequent Remark \ref{rem:nonconsort}.
%

\section{Global well-posedness and decay rate of solutions}
\label{secnonlinear}

Once the linear decay of solutions from the previous section is obtained, we are in a position to perform the nonlinear energy estimate which, in turn, will allow us to close the a priori energy estimate and to get the decay rate of solutions to the nonlinear problem. We get the estimate for the perturbed variable $Z$, which is also valid for the conserved perturbed variables $W-\bW$ as well as for $V-\bV$, in view  of Proposition \ref{Eq-norm-prop}. Here $W=(\rho, \rho u, \rho E, \eta)$ and $V=(\rho, u, \theta, \eta)$, while $\bW=(\brho,\brho\,\bu, \brho\, \overline{E}, \bet) \in \Or_{\mathrm{eq}}$ is the (unique) constant equilibrium state which corresponds to the constant state $\bV=(\brho, \bu, \bthe,\bet)$ satisfying $\bthe^4 = \bet$, and vice versa; in other words, the invertible map $W = W(V)$ maps $\cU_{\mathrm{eq}}$ into $\Or_{\mathrm{eq}}$ back and forth; see Remark \ref{rem:varV} in Section \ref{subsec:yong}. 

\subsection{Nonlinear energy estimate}
Let us start by taking $\bV=(\brho, \bu, \bthe, \bet) \in \cU_{\mathrm{eq}}$, a constant equilibrium state. As previously explained, we can rewrite system \eqref{neq-rad-hyd-1d} linearized around $\bV$ as in \eqref{sys-z-i} in a new set of variables: 
\begin{equation}\label{sys-z-ii}
\tiA^{0}Z_{t}+\tiA^{1}Z_{x}+\tiL Z = \tiB Z_{xx} + \tilde{g}_{x}+ \tilde{q}.
\end{equation}
Let us take the initial data
\[
V_{0}(x) \in H^{s}(\R)\cap L^{1}(\R)
\]
for system \eqref{neq-rad-hyd-1d}, and consider the unique local solution $V(x,t)$ on the time interval $[0,T_{1}]$ satisfying \eqref{reg-loc-ex-1d} and the estimate \eqref{loc-ex-ee-1d} because of Theorem \ref{loc-ex-th-1d}. By Proposition \ref{Eq-norm-prop}, the initial data $V_{0}(x)$ corresponds to an initial data
\begin{equation}\label{in-dat-z}
Z_{0}(x) \in H^{s}(\R) \cap L^{1}(\R), 
\end{equation}
with $s\geq 3$, for the Cauchy problem \eqref{sys-z-ii}-\eqref{in-dat-z}, with solution
\[
Z(x,t) = D_{U}W(\bU) \big( W(V(x,t))- W(\bV) \big),
\]
satisfying the same regularity and energy estimates as those of $V$; see Theorem \ref{loc-ex-th-1d}. Using the Duhamel's formula we can express the variables $Z(x,t)$ as
\[
Z(x,t)= \left( e^{t \Phi}Z_{0} \right)(x) + \int_{0}^{t}\left( e^{(t-\tau)\Phi}\big( \tiA^{0} \big)^{-1}(\tilde{g}_{x} + \tilde{q} \big) \right)(x)\, d\tau,
\]
with the definition of the semigroup $e^{t \Phi}$ given by \eqref{def-sem}. Then, apply Corollary \ref{lin-dec-cor} to obtain
\begin{equation}\label{nl-ee-i}
\begin{aligned}
\Vert \partial_{x}^{\ell}Z(t) \Vert_0 & \leq \Vert \partial_{x}^{\ell} e^{t \Phi}Z_{0} \Vert_0 + \int_{0}^{t}\Vert \partial_{x}^{\ell}\left( e^{(t-\tau)\Phi} \big( \tiA^{0}\big)^{-1} (\tilde{g}_{x} + \tilde{q} ) \right)(\tau) \Vert_0 \, d\tau  \\ & \leq C e^{-c_{1}t} \Vert \partial_{x}^{\ell}Z_{0} \Vert_0 + C (1+t)^{-(\ell/2 + 1/4)}\Vert Z_{0} \Vert_{L^{1}}  \\ & \quad+ \int_{0}^{t}\Vert \partial_{x}^{\ell}\left( e^{(t-\tau)\Phi} \big( \tiA^{0}\big)^{-1} (\tilde{g}_{x} + \tilde{q} )  \right) (\tau) \Vert_0 \, d\tau.
\end{aligned}
\end{equation}
Let us compute the last term on the right-hand side of \eqref{nl-ee-i}.
First, since the identity
\[
\partial_{x}^{\ell}\big( e^{t \Phi}(\partial_{x}f) \big)(x)= \partial_{x}^{\ell +1}\big( e^{t\Phi} f \big)(x)
\]
holds for $f\in H^{s}(\R)$ and $\ell$ such that $0 \leq \ell+1 \leq s$, we have 
\[
\begin{aligned}
\int_{0}^{t} \Vert \partial_{x}^{\ell}\big( e^{(t-\tau)\Phi} \big( \tiA^{0}\big)^{-1} \tilde{g}_{x}    \big) (\tau) \Vert_0 \, d\tau &= \int_{0}^{t}\Vert \partial_{x}^{\ell+1}\big( e^{(t-\tau)\Phi} \big( \tiA^{0}\big)^{-1} \tilde{g}    \big) (\tau) \Vert_0 \, d\tau   \\ & \leq  C\int_{0}^{t}e^{-c_{1}(t-\tau)}\Vert \partial_{x}^{\ell + 1} \tilde{g} (\tau) \Vert_0 \, d\tau \\ 
&\ + C\int_{0}^{t} (1+(t-\tau))^{-(\ell/2 + 3/4)}\Vert \tilde{g}(\tau)\Vert_{L^{1}}\, d \tau.
\end{aligned}
\]
For the term involving $\tilde{q}$, we are going to use that $\tilde{q} \in \mathcal{M}^{\bot}$, where $\mathcal{M}$ is the null space of $\tiL$. Thus Lemma \ref{lin-dec-shr-lem} implies the estimate
\[
\begin{aligned}
\int_{0}^{t} \Vert \partial_{x}^{\ell}\big( e^{(t-\tau)\Phi} \big( \tiA^{0}\big)^{-1} \tilde{q}    \big) (\tau) \Vert_0 \, d\tau  &\leq  C\int_{0}^{t}e^{-c_{1}(t-\tau)}\Vert \partial_{x}^{\ell}\tilde{q}(\tau) \Vert_0 \, d\tau + \\ & \quad + C\int_{0}^{t} (1+(t-\tau))^{-(\ell/2 + 3/4)}\Vert \tilde{q}(\tau)\Vert_{L^{1}}\, d \tau.
\end{aligned}
\]
Substituting the last two estimates into \eqref{nl-ee-i} yields
\[
\begin{aligned}
\Vert \partial_{x}^{\ell} Z(t) \Vert_0 &\leq Ce^{-c_{1}t}\Vert \partial_{x}^{\ell} Z_{0} \Vert_0 + C(1+t)^{-(\ell/2 +1/4)}\Vert Z_{0} \Vert_{L^{1}}   \\ & \quad + C\int_{0}^{t}e^{-c_{1}(t-\tau)}\left( \Vert \partial_{x}^{\ell + 1}\tilde{g}(\tau) \Vert_0 + \Vert \partial_{x}^{\ell}\tilde{q}(\tau) \Vert_0 \right)\, d \tau  \\ &\quad + C\int_{0}^{t}(1+ (t-\tau))^{-(\ell/2+3/4)} \left( \Vert \tilde{g}(\tau) \Vert_{L^{1}} + \Vert \tilde{q} (\tau) \Vert_{L^{1}} \right)\, d \tau. 
\end{aligned}    
\]
Summing up this last estimate for $\ell=0,\ldots,s-1$ we obtain
\begin{equation}
\label{nl-ee-iia}
\begin{aligned}
\Vert  Z(t) \Vert_{s-1} &\leq Ce^{-c_{1}t}\Vert  Z_{0} \Vert_{s-1} + C(1+t)^{-1/4}\Vert Z_{0} \Vert_{L^{1}} 
\\ & \quad + C\int_{0}^{t}e^{-c_{1}(t-\tau)}\left( \Vert \tilde{g}(\tau) \Vert_{s} + \Vert \tilde{q}(\tau) \Vert_{s-1} \right)\, d \tau 
\\ &\quad + C\int_{0}^{t}(1+ (t-\tau))^{-3/4} \left( \Vert \tilde{g}(\tau) \Vert_{L^{1}} + \Vert \tilde{q} (\tau) \Vert_{L^{1}} \right)\, d \tau. 
\end{aligned}    
\end{equation}

Next we estimate the Sobolev and $L^{1}$ norms of $\tilde{g}$ and $\tilde{g}$ appearing  on the right-hand side of  \eqref{nl-ee-iia}. For this purpose, let us remember that (see \eqref{order-pq})
\[
\begin{aligned}
\tilde{g} & = O \big( \vert W - \bW \vert^{2} \big), \\
\tilde{q} & = O \big( \vert W - \bW \vert \vert \eta_{x} \vert + \vert W - \bW \vert^{2} \big). 
\end{aligned}
\]
Since we are assuming $s\geq 3$, we can use the Banach algebra properties of $H^{s}(\R)$ and the Sobolev calculus inequalities (see, e.g., Theorem 7.77 in \cite{IoIo01} and Lemma 3.2 in \cite{HaLi96a}) to obtain
\[
\begin{aligned}
\Vert \vert W - \bW \vert^{2} \Vert_{s} \leq C \Vert W- \bW \Vert_{s}\Vert W - \bW \Vert_{L^{\infty}} &\leq C \Vert W - \bW \Vert_{s} \Vert W - \bW \Vert_{1}  \\ &\leq C \Vert W - \bW \Vert_{s}\Vert W - \bW \Vert_{s-1}, 
\end{aligned}
\]
and
\[
\begin{aligned}
\Vert \vert W - \bW \vert \vert \eta_{x} \vert \Vert_{s-1} & \leq C \left( \Vert W -\bW \Vert_{s-1}\Vert \eta_{x} \Vert_{L^{\infty}} + \Vert \eta_{x} \Vert_{s-1}\Vert W -\bW \Vert_{L^{\infty}}  \right) \\ & \leq C \left( \Vert W -\bW \Vert_{s-1}\Vert \eta_{x} \Vert_{1} + \Vert \eta_{x} \Vert_{s-1}\Vert W -\bW \Vert_{1} \right) \\ &\leq C  \Vert W -\bW \Vert_{s-1}\Vert \eta_{x} \Vert_{s-1},
\end{aligned}
\]
so that we have
\begin{equation}\label{nl-tr-norm}
\begin{aligned}
\Vert \tilde{g}(\tau) \Vert_{s} \leq C \Vert (W - \bW)(\tau) \Vert_{s-1}\Vert (W - \bW)(\tau) \Vert_{s}, \\
\Vert \tilde{q}(\tau) \Vert_{s-1} \leq C \Vert (W - \bW)(\tau) \Vert_{s-1}\Vert (W-\bW)(\tau) \Vert_{s}. 
\end{aligned}
\end{equation}
In addition, it is easy to verify that
\begin{equation}\label{nl-tr-norm-L1}
\begin{aligned}
\Vert \tilde{g}(\tau) \Vert_{L^{1}} &\leq C \Vert (W-\bW)(\tau) \Vert^{2}_{0} \leq C \Vert (W- \bW)(\tau) \Vert_{s-1}^{2}, \\
\Vert \tilde{q}(\tau) \Vert_{L^{1}} & \leq C \Vert (W - \bW)(\tau) \Vert_{1}^{2} \leq C \Vert (W-\bW)(\tau) \Vert_{s-1}^{2}. 
\end{aligned}
\end{equation}
Thus, using estimates \eqref{nl-tr-norm} and \eqref{nl-tr-norm-L1} in estimate \eqref{nl-ee-iia}, we are lead to
\begin{equation}\label{nl-ee-iii}
\begin{aligned}
\Vert  Z(t) \Vert_{s-1} &\leq Ce^{-c_{1}t}\Vert  Z_{0} \Vert_{s-1} + C(1+t)^{-1/4}\Vert Z_{0} \Vert_{L^{1}}   \\ & \quad + C\int_{0}^{t}e^{-c_{1}(t-\tau)} \Vert (W-\bW)(\tau) \Vert_{s-1}\Vert (W-\bW)(\tau) \Vert_{s} \, d \tau 
\\ &\quad + C\int_{0}^{t}(1+ (t-\tau))^{-3/4} \Vert (W -\bW)(\tau) \Vert_{s-1}^{2} \, d \tau. 
\end{aligned}    
\end{equation}
Next we apply Proposition \ref{Eq-norm-prop} to recast the estimate above in terms of the original perturbation variables $V-\bV$: 
\begin{equation}\label{nl-ee-V}
\begin{aligned}
\Vert  (V-\bV)(t) \Vert_{s-1} &\leq C(1+t)^{-1/4} \left( \Vert  V_{0}-\bV \Vert_{s-1} +\Vert V_{0}-\bV \Vert_{L^{1}} \right) 
\\ & \quad + C\sup_{0 \leq \tau \leq t}\Vert (V-\bV)(\tau) \Vert_{s} \int_{0}^{t}e^{-c_{1}(t-\tau)} \Vert (V-\bV)(\tau) \Vert_{s-1} \, d \tau 
\\ &\quad + C\int_{0}^{t}(1+ (t-\tau))^{-3/4} \Vert (V -\bV)(\tau) \Vert_{s-1}^{2} \, d \tau. 
\end{aligned}    
\end{equation}
Now, let us  define 
\[
\vertiii{(V-\bV)(t)}_{s}:= \sup_{0 \leq \tau \leq t} (1+t)^{1/4} \Vert (V-\bV)(\tau) \Vert_{s-1}. 
\]
With this notation, from estimate \eqref{nl-ee-V} we obtain
\begin{equation}\label{alm-nl-dec-i}
\begin{aligned}
\vertiii{(V-\bV)(t)}_{s} &\leq C \left( \Vert V_{0}-\bV \Vert_{s-1} + \Vert V_{0}- \bV \Vert_{L^{1}} \right) + C \mu_{1}(t) \vertiii{(V - \bV)(t)}_{s}E_{s}(t) 
\\ & \quad + C \mu_{2}(t)\vertiii{(V-\bV)(t)}_{s}^{2},  
\end{aligned}
\end{equation}
where $\mu_{1}(t)$ and $\mu_{2}(t)$ are given by
\[
\begin{aligned}
\mu_{1}(t) &:= \sup_{0 \leq \tau \leq t}(1+\tau)^{1/4}\int_{0}^{\tau}e^{-c_{1}(\tau-\tau_{1})}(1+\tau_{1})^{-1/4}\, d \tau_{1},  \\
\mu_{2}(t) &:= \sup_{0 \leq \tau \leq t} (1+\tau)^{1/4}\int_{0}^{\tau} (1+\tau-\tau_{1})^{-3/4}(1+\tau_{1})^{-1/2}\, d\tau_{1}. 
\end{aligned}
\]
As $\mu_{1}(t)$ and $\mu_{2}(t)$ are uniformly bounded in $t$ (see, for example, Lemma A.1 in \cite{PlV22}), we can rewrite \eqref{alm-nl-dec-i} as
\begin{equation}\label{alm-nl-dec-ii}
\begin{aligned}
\vertiii{(V-\bV)(t)}_{s} &\leq C \left( \Vert V_{0}-\bV \Vert_{s-1} + \Vert V_{0}- \bV \Vert_{L^{1}} \right) + C  \vertiii{(V - \bV)(t)}_{s}E_{s}(t)  \\ &\quad + C \vertiii{(V-\bV)(t)}_{s}^{2}. 
\end{aligned}
\end{equation}
Hence we have proved the following result. 

\begin{proposition}
\label{nl-dec-prop}
Assume hypotheses \hyperref[H1]{\rm{(H$_1$)}} -- \hyperref[H3]{\rm{(H$_3$)}}. Let $V(x,t)$ be the local solution on the time interval $[0,T]$ of the Cauchy problem for system \eqref{neq-rad-hyd-1d} with initial data $V_{0}(x)$ satisfying $V_{0}-\bV \in \big(H^{s}(\R) \cap L^{1}(\R)\big)^4$, $s\geq 3$, and define
\[
\Vert V_{0} - \bV \Vert_{k,1}:= \Vert V_{0}-\bV \Vert_{k} + \Vert V_{0} - \bV \Vert_{L^{1}},
\]
for $0 \leq k \leq s$. Then there exist positive constants $a_{1} (\leq a_{0})$ (with $a_{0}$ as in Lemmata \ref{lem-apr-ee-1} and \ref{lem-apr-ee-2}) and $\delta_{1}=\delta_{1}(a_{1})$ such that if $N_{s}(T)\leq a_{1}$ and $\Vert V_{0}- \bV \Vert_{s-1,1}\leq \delta_{1}$, then the estimate
\begin{equation}\label{nl-dec}
\Vert (V-\bV)(t) \Vert_{s-1} \leq C_{3} (1+t)^{-1/4}\Vert V_{0} - \bV \Vert_{s-1,1}    
\end{equation}
holds for all $t \in [0,T]$, with a positive constant $C_{3}=C_{3}(a_{1},\delta_{1})$. 
\end{proposition}
As a consequence of the proposition above we obtain the a priori energy estimate that we were missing. 
\begin{corollary}\label{cor-miss-apr-ee}
Under the same assumptions of Proposition \ref{nl-dec-prop}, the estimate 
\begin{equation}\label{miss-apr-ee}
\Vert (V - \bV)(\tau) \Vert_{s-1}^{2} + \int_{0}^{t} \Vert \partial_{x}V (\tau) \Vert_{s-2}^{2}\, d \tau \leq C_{4} \Vert V_{0} - \bV \Vert_{s-1,1}^{2} 
\end{equation}
holds for all $t \in [0,T]$ and some positive constant $C_{4}=C_{4}(a_{1},\delta_{1})$.
\end{corollary}
\begin{proof}
We follow the same steps to get estimate \eqref{nl-dec}. At the point where we obtain estimate \eqref{nl-ee-iia}, we sum up from $\ell = 1$ to $\ell = s-1$ to obtain the estimate
\[
\Vert \partial_{x}V(\tau) \Vert_{s-2} \leq C_{1}(1+t)^{-3/4} \Vert V_{0}- \bV \Vert_{s-1,1}.
\]
Then, by taking the square of this estimate and integrating on time we arrive at
\[
\int_{0}^{t} \Vert \partial_{x} V (\tau) \Vert_{s-2}^{2}\, d \tau \leq C \Vert V_{0} - \bV \Vert_{s-1,1}^{2},
\]
where $C$ is a positive constant uniform in $t$. Here we have used that the function $(1+t)^{-3/2}$ is integrable in $[0,\infty)$. Finally, combining the last estimate and estimate \eqref{nl-dec} we obtain the result.  
\end{proof}

\subsection{Global decay rate of small perturbations and proof of Theorem \ref{thmgloex}}
Up to this point, we are almost ready to prove the main result of the paper: the global existence and asymptotic decay of small perturbations of constant state solutions to   system \eqref{neq-rad-hyd-1d}. For this, we only need the appropriate a priori energy estimate to perform the standard continuation argument of the local solution. This estimate is a direct consequence of Lemmata \ref{lem-apr-ee-1} and \ref{lem-apr-ee-2} and Corollary \ref{cor-miss-apr-ee}, and it is the content of the next result.

\begin{corollary}\label{fin-apr-ee-cor}
Let $V(x,t)$ be the local solution on $[0,T]$ of the initial value problem of system \eqref{neq-rad-hyd-1d} with initial data $V_0$ satisfying $V_{0}-\bV \in \big(H^{s}(\R) \cap L^{1}(\R)\big)^4$, $s \geq 3$, with the regularity \eqref{reg-loc-ex-1d} from the local existence Theorem \ref{loc-ex-th-1d}. Then there exist positive constants $a_{2} (\leq a_{1})$ and $C_{5}=C_{5}(a_{2},\delta_{1})$, with $a_{1}$ and $\delta_{1}$ as in Proposition \ref{nl-dec-prop}, such that if $N_{s}(T)\leq a_{2}$ and $\Vert V_{0}- \bV \Vert_{s-1,1}\leq \delta_{1}$, then the estimate 
\begin{equation}
\label{fin-apr-ee}
\sup_{0 \leq \tau \leq t} \Vert (V-\bV)(\tau) \Vert_{s}^{2} + \int_{0}^{t} \Vert \partial_{x}(\rho, u, \theta)(\tau) \Vert_{s-1}^{2} + \Vert \partial_{x}\eta(\tau) \Vert_{s}^{2}\, d \tau \leq C_{5} \Vert V_{0} - \bV \Vert_{s,1}^{2},
\end{equation}
holds for all $t \in [0,T]$. 
\end{corollary}
\begin{proof}
Combine the estimates given by Lemmata \ref{lem-apr-ee-1} and \ref{lem-apr-ee-2} and that of Corollary \ref{cor-miss-apr-ee} in the form \eqref{apr-ee-1} + $\alpha$ \eqref{apr-ee-2} + \eqref{miss-apr-ee} for some $\alpha>0$ satisfying $\alpha C_{2} < 1$. Thus we are lead to
\[
\begin{aligned}
N_{s}(t)^{2} &\leq \sup_{0\leq \tau\leq t}\Vert (V-\bV)(\tau) \Vert_{s-1}^2+\sup_{0\leq \tau\leq t}\Vert \partial_{x}V(\tau) \Vert_{s-1}^{2} 
 \\ &\quad + \int_{0}^{t}\big( \Vert \partial_{x}^{2}(\rho, u,\theta)(\tau)\Vert_{s-2}^{2}+ \Vert \partial_{x}^{2}\eta(\tau)\Vert_{s-1}^{2}\big)d\tau 
    \\&\quad + \int_{0}^{t}\Vert \partial_{x}V(\tau) \Vert_{s-2}^{2}\, d \tau 
     \\ & \leq C \left( \Vert V_{0}- \bV \Vert_{s,1}^{2} + N_{s}(T)^{3} \right), 
\end{aligned}
\]
from which we get \eqref{fin-apr-ee} by taking $N_{s}(T)$ small enough. 
\end{proof}

\begin{proof}[Proof of Theorem \ref{thmgloex}]

Follows directly from the local existence Theorem \ref{loc-ex-th-1d}, Proposition \ref{nl-dec-prop}, Corollary \ref{fin-apr-ee-cor} and a standard continuation argument; for further details, see the proof of Theorem 6.1 in \cite{PlV22} or that of Theorem 5.1 in \cite{PlV2}.

\end{proof}

\section{Discussion}
\label{secdicussion}

In this paper we have studied a non-equilibrium diffusion limit system of equations, derived by Buet and Despr\'es \cite{BuDe04}, which describes the dynamics of a non-relativistic, strongly radiative inviscid fluid. The system can be viewed as a singular limit in the non-equilibrium diffusion regime, that is, when the temperature of radiation is a priori different from the fluid temperature. The radiation appears through an extra equation of parabolic type for the radiative temperature. This parabolic term, together with the relaxation term (which also comes from the equation for radiative energy), are the only dissipative mechanism within the system. In previous works, damping, viscous or heat conduction effects for the fluid have been incorporated into the equations in order to show global existence of perturbations of (and stability of) constant states. Up to our knowledge, this is the first contribution addressing this issue for the original (inviscid, non-heat-conducting and without damping) set of equations proposed in \cite{BuDe04}. We show the global existence and decay in time of perturbations of constant equilibrium states for the system in one space dimension.

Instead of working directly with the model equations and of performing \emph{ad hoc} energy estimates, we adopted an abstract methodology that involves the strict dissipativity and the linear decay structure of the system, which can be extrapolated to the nonlinear problem. This method falls under the framework of the classical work by Kawashima and Shizuta \cite{KaSh88a,ShKa85} for systems of hyperbolic-parabolic type. For this purpose, we proved that the entropy function identified by Buet and Despr\'es \cite{BuDe04} can be used to symmetrize the system and to recast the problem in terms of new perturbation variables. The latter and the resulting system of equations play a crucial role in the establishment of the linear decay of the associated semigroup, based on the genuine coupling condition in one space dimension. In addition, we proved that for dimensions $d \geq 2$ the system fails to be genuinely coupled; see Proposition \ref{propnongc} below. This fact justifies the application of the methods to the system in one dimension only. The lack of dissipative terms such as material viscosities, the lack of damping terms, as well as the non-genuinely coupled nature of the system, make the multi-dimensional model worthy of further investigations.


\section*{Acknowledgements}

J. M. Valdovinos is grateful to the Department of Engineering, Information Sciences and Mathematics of the University of L'Aquila for their hospitality during a research visit when this work was initiated.
The work of C. Lattanzio is partially supported by INdAM--GNAMPA. The work of R. G. Plaza was supported by SECIHTI, Mexico, grant ``Ciencia de Frontera'' CF-2023-G-122. The work of J. M. Valdovinos was supported by SECIHTI, Mexico, Program ``Ayudantes de Investigaci\'on 2025" of the Sistema Nacional de Investigadoras e Investigadores.

\appendix
\section{Local existence in several space dimensions}
\label{app1}

In this appendix we state the local (in time) well-posedness for the Cauchy problem of the non-equilibrium system \eqref{neq-rad-hyd} in any space dimension $d \geq 1$ in the perturbation framework around a constant equilibrium state, provided that the initial data is sufficiently smooth and close to the constant equilibrium state. This local existence result follows directly from a previous theorem by Kawashima \cite{KaTh83}. Here we verify that the generic radiation system \eqref{neq-rad-hyd}  can be put into the non-homogeneous quasilinear form needed in \cite{KaTh83}. For this purpose, first we use  the continuity equation to simplify the momentum equation as follows:  
\[
\rho \partial_{t}\bm{u} + \rho(\nabla \bm{u})\bm{u} + \nabla \big( p + \tfrac{1}{3}\eta) = 0.
\]
Next, take the inner product of the resulting momentum equation with the velocity field $\bm{u}$ to obtain
\begin{equation}\label{simp-mom-eq}
\rho \partial_{t}\big( \tfrac{1}{2} \vert \bm{u} \vert^{2} \big) + \rho \nabla \big( \tfrac{1}{2}\vert \bm{u} \vert^{2} \big) \cdot \bm{u} + \nabla p \cdot \bm{u} = - \tfrac{1}{3}\nabla \eta  \cdot \bm{u}. 
\end{equation} 
Let us recast the energy equations as
\[
\partial_{t}( \rho E ) + \nabla \cdot \big( \big( \rho E  + p  \big) \bm{u} \big)= \nabla \cdot \big( \tfrac{1}{3\sigma_{s}}\nabla \eta \big)- \big( \partial_{t}\eta + \nabla \cdot \big( \eta\bm{u} + \tfrac{1}{3}\eta \bm{u} \big) \big).
\]
The latter can be simplified, using the continuity equation, the equation for the radiation intensity and the fact that $E=e+\tfrac{1}{2}\vert \bm{u} \vert^{2}$, into  
\[
\rho \partial_{t} \big( e + \tfrac{1}{2}\vert \bm{u} \vert^{2} \big) + \rho \nabla \big( e + \tfrac{1}{2} \vert \bm{u} \vert^{2} \big) \cdot \bm{u} + \nabla p \cdot \bm{u} + p \nabla \cdot \bm{u} = \sigma_{a}( \eta- \theta^{4} ) - \tfrac{1}{3}\nabla \eta \cdot \bm{u}. 
\]
Now subtract equation \eqref{simp-mom-eq} from the last equation to obtain
\[
\rho \partial_{t}  e  + \rho \nabla  e \cdot \bm{u}+ p \nabla \cdot \bm{u} = \sigma_{a}( \eta- \theta^{4} ).
\]
In this last equation we use again the continuity equation and   $e_{t}=e_{\rho}\rho_{t} + e_{\theta}\theta_{t}$ and $\nabla e = e_{\rho}\nabla \rho + e_{\theta}\nabla \theta$ (we are assuming that the internal energy is a function of the density and temperature) to arrive at
\[
\rho e_{\theta}\big( \partial_{t}\theta +  \nabla \theta \cdot \bm{u} \big) + \big( p - \rho^{2}e_{\rho} \big) \nabla \cdot \bm{u} = \sigma_{a}( \eta - \theta^{4}).
\]
Finally,  using the thermodynamic relation $p- \rho^{2}e_{\rho}= \theta p_{\theta}$ in the last expression, we arrive to the following 
 quasilinear form of system \eqref{neq-rad-hyd} 
\begin{equation}\label{neq-rad-hyd-quas}
A^{0}(V)V_{t} + \sum_{j=1}^{d}A^{j}(V)V_{x_j} + Q(V) = \sum_{i,j=1}^{d}B^{jk}(V)V_{x_{j}x_{k}},
\end{equation}
where $V = (\rho, \bm{u}, \theta, \eta) \in \R^{d+3}$, $\bm{u}=(u_{1}, \ldots, u_{d}) \in \R^d$, and
\begin{equation}\label{A0-Q}
A^{0}(V) = \begin{pmatrix}
1 & 0_{1\times d} & 0 & 0 \\ 0_{d\times 1} & \rho\: I_{d} & 0_{d \times 1} & 0_{d\times 1} \\ 0 & 0_{1\times d} & \rho \: e_{\theta} & 0 \\ 0 & 0_{1\times d} & 0 & 1
\end{pmatrix},\quad Q(V)= \begin{pmatrix}
0 \\ 0_{d\times 1} \\ \sigma_{a}(\theta^{4} -\eta) \\ \sigma_{a}(\eta-\theta^{4}) 
\end{pmatrix}.
\end{equation}
The first- and second-order matrix coefficients are given in terms of their symbols: 
\begin{equation}\label{sym-A}
\sum_{j=1}^{d}A^{j}(V)\xi_{j} = \begin{pmatrix}
\bm{u}\cdot \xi & \rho \: \xi  & 0 & 0 \\ p_{\rho} \: \xi^{\top} & \rho\: (\bm{u}\cdot \xi) \: I_{d} & p_{\theta}\: \xi^{\top}  & \frac{1}{3}\xi^{\top} \\ 0 & \theta\: p_{\theta}\: \xi & \rho\: e_{\theta}\: (\bm{u}\cdot \xi) & 0 \\ 0 & \frac{4}{3}\eta\: \xi & 0 & \bm{u}\cdot \xi
\end{pmatrix}
\end{equation}
 and 
\begin{equation}\label{sym-B}
\sum_{j,k=1}^{d}B^{jk}(V)\xi_{j}\xi_{k}= \begin{pmatrix}
0 & 0_{1\times d} & 0 & 0 \\ 0_{d\times 1} & 0 \: I_{d} & 0_{d\times 1} & 0_{d\times 1} \\ 0 & 0_{1\times d} & 0 & 0 \\ 0 & 0_{1\times d} & 0 & \frac{\vert \xi \vert^{2}}{3\sigma_{s}}
\end{pmatrix}
\end{equation}
for $\xi=(\xi_{1}, \ldots,\xi_{d})\in \R^{3}$.

In order to state the local existence of solutions, we multiply the first and second equations in \eqref{neq-rad-hyd-quas} by $\theta\, p_{\rho}/\rho$ and $\theta$, respectively, to obtain
\begin{equation}\label{quas-hyp-par-form}
\begin{aligned}
 A_{1}^{0}(V)\begin{pmatrix}
\rho \\ \bm{u} \\ \theta
\end{pmatrix}_{t} + \sum_{j=1}^{d}A_{11}^{j}(V)\begin{pmatrix}
\rho \\ \bm{u} \\ \theta
\end{pmatrix}_{x_{j}} &= f_{1}(V, \nabla \eta), 
\\ \eta_{t} - \tfrac{1}{3\sigma_{s}} \Delta \eta &= f_{2}(V,\nabla V), 
\end{aligned}
\end{equation}
where the matrix coefficients are given by
\[
A_{1}^{0}(V) = \begin{pmatrix}
\theta\, p_{\rho}/\rho & 0_{1\times d} & 0 \\ 0_{d\times 1} & \rho\,\theta \, I_{d} & 0_{d\times 1} \\ 0 & 0_{1\times d} & \rho\,e_{\theta} 
\end{pmatrix},
\]
\[
A_{11}^{j}(V) = \begin{pmatrix}
\theta\,p_{\rho} (\bm{u}\cdot \hat{e}_{j})/\rho & \theta\, p_{\rho}\hat{e}_{j}& 0 \\ \theta\,p_{\rho}\hat{e}_{j}^{\top} & \rho\,\theta\,(\bm{u}\cdot \hat{e}_{j}) I_{d} & \theta\,p_{\theta}\hat{e}_{j}^{\top} \\  0 & \theta\,p_{\theta}\hat{e}_{j} & \rho\,e_{\theta} (\bm{u}\cdot \hat{e}_{j})
\end{pmatrix}, 
\]
and the non homogeneous terms take the form
\[
f_{1}(V,\nabla \eta)= \begin{pmatrix}
0 \\ -\theta (\nabla \eta )^{\top}/3 \\ \sigma_{a}(\eta-\theta^{4})
\end{pmatrix},
\]
\[
f_{2}(V, \nabla V)=\sigma_{a}(\theta^{4}-\eta) -\nabla(\eta \bm{u}) - \tfrac{1}{3}\eta \nabla \cdot \bm{u}. 
\]
In the expression above, $\nabla V$ denotes the tensor containing all the derivatives of order one of $V=(\rho, \bm{u},\theta, \eta)$.

We are interested in the initial value problem for system \eqref{quas-hyp-par-form} with initial data
\begin{equation}\label{in-dat} 
V(x,0)=V_{0}(x)=(\rho_{0},\bm{u}_{0},\theta_{0},\eta_{0})(x).
\end{equation}
System \eqref{quas-hyp-par-form} falls into the general class of quasilinear symmetric hyperbolic-parabolic system of composite type for which Kawashima (see \cite[Section 2.1]{KaTh83}) proved the local well-posedness. Thus we have the following theorem, which is a restatement of Theorem 2.9 in \cite{KaTh83}. 

\begin{theorem}\label{loc-ex-th}
Let $\bV=(\brho, \bbu, \bthe, \bet) \in \R^{d+3}$ be a constant equilibrium state. Consider the Cauchy problem for system \eqref{quas-hyp-par-form} with initial data $V_{0}(x)$ such that $V_{0}-\bV \in H^{s}(\R^{d})$, $s\geq s_{0}+1$, with $s_{0}:=\lfloor \tfrac{d}{2} \rfloor+1$. Then there exists $\epsilon > 0$ such that if
\[
a_{0}:= \Vert V_{0}-\bV \Vert_{s} \leq \epsilon,
\]
we have that $m_{1}\leq \rho_{0}(x) \leq M_{1}$, $m_{2}\leq \theta_{0}(x) \leq M_{2}$, $m_{3}\leq \eta_{0}(x) \leq M_{3}$ for all $x\in \R^d$ and for some positive constants $0 < m_{i} < M_{i}$, $i=1,2,3$, and there exists $T_{0}=T_{0}(a_{0}) > 0$ such that the Cauchy problem has a unique solution $V = (\rho, \bm{u},\theta, \eta)$ satisfying
\begin{equation*}
\begin{aligned}
&\rho-\brho, \bm{u}-\bbu, \theta-\bthe \in C\left( [0,T_{0}]; H^{s}(\R^{d}) \right)    \cap C^{1}\left( [0,T_{0}]; H^{s-1}(\R^{d}) \right),\\
& \eta-\bet \in C\left( [0,T_{0}]; H^{s}(\R^{d}) \right) \cap C^{1}\left( [0,T_{0}]; H^{s-2}(\R^{d}) \right),
\end{aligned}
\end{equation*}
and the estimate
\begin{equation*}
\sup_{0 \leq \tau \leq t} \Vert (V-\bV)(\tau) \Vert_{s}^{2} + \int_{0}^{t}  \Vert \nabla (\rho, \bm{u}, \theta)(\tau)\Vert_{s-1}^{2}+ \Vert \nabla \eta(\tau) \Vert_{s}^{2}\, d \tau \leq C_{0} \Vert V_{0}-\bV \Vert_{s}^{2},
\end{equation*}
holds for all $t\in [0,T_{0}]$, and some positive constant $C_{0}$ depending on $\Vert V_{0}- \bV \Vert_{s}$.
\end{theorem}

\section{Non-genuine coupling in dimension $d \geq 2$}
\label{non-gen-cou-md}
In this section we prove that system \eqref{neq-rad-hyd} (at the linear level) does not satisfy the genuine coupling condition when the space dimension is $d \geq 2$. For this purpose, let us first write the linearized system around a constant equilibrium state $\bV=(\brho, \bbu, \bthe, \bet)$, with $\bet = \bthe^{4}$. Based on the calculation of its quasilinear form \eqref{neq-rad-hyd-quas}, the linearized system around $\bV$ reads
\begin{equation}
\label{lin-sys-md}
A^{0}V_{t} + \sum_{j=1}^{d}A^{j}V_{x_j}+ LV = \sum_{j,k=1}^{d}B^{jk}V_{x_j x_k},
\end{equation} 
where  $A^0$, $A^{j}$ and $B^{jk}$ are the matrices given in \eqref{A0-Q}-\eqref{sym-B} evaluated at $\bV$, and $L=D_{V}Q(\bV)$. More precisely,
\[
A^{0}=A^0(\bV)=\begin{pmatrix}
1 & 0_{1\times d} & 0 & 0 \\ 0_{d\times 1} & \brho I_{d} & 0_{d\times 1} & 0_{d\times 1} \\ 0 & 0_{1\times d} & \brho\, \be_{\theta} & 0 \\ 0 & 0_{1 \times d} & 0 & 1
\end{pmatrix},
\]
\[
L =D_{V}Q(\bV)=\begin{pmatrix}
0 & 0_{1\times d} & 0 & 0 \\ 0_{d\times 1} & 0_{d} & 0_{d\times 1} & 0_{d\times 1} \\ 0 & 0_{1\times d} & 4\sigma_{a}\bthe^3 & -\sigma_{a} \\ 0 & 0_{1 \times d} & -4\sigma_{a}\bthe^3 & \sigma_{a}
\end{pmatrix},
\]
\[
A^{j}=A^j(\bV)=\begin{pmatrix}
\bbu\cdot \hat{e}_{j} & \brho\,\hat{e}_{j} & 0 & 0\\ \bp_{\rho} \hat{e}_{j}^{\top} & \brho\, (\bbu\cdot \hat{e}_{j}) I_d & \bp_{\theta}\hat{e}_{j}^{\top} & \frac{1}{3}\hat{e}_{j}^{\top}\\ 0 & \bthe\,\bp_{\theta}\hat{e}_{j} & \brho\, \be_{\theta} (\bbu\cdot \hat{e}_{j}) & 0 \\ 0 & \frac{4}{3}\bet\, \hat{e}_{j} & 0 & \bbu\cdot \hat{e}_{j}
\end{pmatrix},
\]
\[
B^{jk}=B^{jk}(\bV)=\begin{pmatrix}
0 & 0_{1\times d} & 0 & 0 \\ 0_{d\times 1} & 0_{d} & 0_{d\times 1} & 0_{d\times 1} \\ 0 & 0_{1\times d} & 0 & 0 \\ 0 & 0_{1 \times d} & 0 & \frac{\delta_{jk}}{3\sigma_{s}}
\end{pmatrix}.
\]
System \eqref{lin-sys-md} is not in symmetric form, but using the relation $\bet=\bthe^{4}$ one can easily see that
\[
S=\begin{pmatrix}
\bthe\, \bp_{\rho}\bet/\brho & 0_{1\times d} & 0 & 0\\ 0_{d \times 1} & \bthe\, \bet \, I_d & 0_{1\times d} & 0_{1 \times d}\\ 0 & 0_{1\times d} & \bet & 0 \\ 0 & 0_{1\times d} & 0 & \bthe/4
\end{pmatrix}
\]
is a symmetrizer. Indeed, we have
\begin{equation}\label{A0-b}
\bA^0:=SA^{0}=\begin{pmatrix}
\bthe\,\bp_{\rho}\bet/\brho & 0_{1\times d} & 0 & 0 \\ 0_{d\times 1} & \brho\,\bthe\,\bet I_d & 0_{d\times 1} & 0_{d\times 1} \\ 0 & 0_{1\times d} & \brho\,\be_{\theta}\bet & 0 \\ 0 & 0_{1\times d} & 0 & \bthe/4
\end{pmatrix},
\end{equation}
\begin{equation}\label{L-b}
\bL:=SL=\begin{pmatrix}
0 & 0_{1\times d} & 0 & 0 \\ 0_{d\times 1} & 0_d & 0_{d\times d} & 0_{d\times d} \\ 0 & 0_{1\times d} & 4\sigma_{a}\bthe^3\bet & -\sigma_{a}\bet \\ 0 & 0_{1\times d} & -\sigma_{a}\bthe^4 & \sigma_{a}\bthe/4
\end{pmatrix} = \begin{pmatrix}
0 & 0_{1\times d} & 0 & 0 \\ 0_{d\times 1} & 0_d & 0_{d\times d} & 0_{d\times d} \\ 0 & 0_{1\times d} & 4\sigma_{a}\bthe^3\bet & -\sigma_{a}\bthe^4 \\ 0 & 0_{1\times d} & -\sigma_{a}\bthe^4 & \sigma_{a}\bthe/4
\end{pmatrix},
\end{equation}
 using the relation $\bet=\bthe^4$, and  
\begin{equation}\label{Aj-b}
\bA^j:=SA^j=\begin{pmatrix}
\bthe\,\bp_{\rho}\bet\, (\bbu\cdot \hat{e}_{j})/\brho & \bthe\,\bp_{\rho}\bet\,\hat{e}_{j} & 0 & 0 \\ \bthe\,\bp_{\rho}\bet \hat{e}_{j}^{\top} & \brho\,\bthe\,\bet (\bbu\cdot \hat{e}_{j}) I_d & \bthe\,\bet\,\bp_{\theta}\hat{e}_{j}^{\top} & \frac{1}{3}\bthe\,\bet\, \hat{e}_{j}^{\top} \\ 0 & \bthe\,\bet\,\bp_{\theta}\hat{e}_{j} & \brho\, \be_{\theta}\bet\,(\bbu\cdot \hat{e}_{j}) & 0 \\ 0 & \frac{1}{3}\bthe\,\bet\, \hat{e}_{j} & 0 & \bthe\, (\bbu\cdot \hat{e}_{j})/4
\end{pmatrix},
\end{equation}
\begin{equation}\label{Bjk-b}
\bB^{jk}:=SB^{jk}= \begin{pmatrix}
0 & 0_{1\times d} & 0 & 0 \\ 0_{d\times 1} & 0_{d} & 0_{d\times 1} & 0_{d\times 1} \\ 0 & 0_{1\times d} & 0 & 0 \\ 0 & 0_{1 \times d} & 0 & \frac{\delta_{jk}\bthe}{12\sigma_{s}}
\end{pmatrix}.
\end{equation}

Hence, if we multiply system \eqref{lin-sys-md} by $S$ on the left we arrive at the following symmetric constant coefficient system,
\begin{equation}
\label{symm-linedsyst}
\bA^{0}V_{t} + \sum_{j=1}^{d}\bA^{j}V_{x_j}+ \bL V = \sum_{j,k=1}^{d}\bB^{jk}V_{x_j x_k}.
\end{equation}

\begin{proposition}
\label{propnongc}
If $d \geq 2$, then system \eqref{symm-linedsyst} is not genuinely coupled.
\end{proposition}
\begin{proof}
Applying the Fourier transform to  \eqref{symm-linedsyst} we end up to
\[
\bA^0\hV_{t}+ \big( i \vert \xi \vert A(\omega) + \bL + \vert \xi \vert^2 B(\omega) \big) \hV=0,
\]
where $\omega = \xi/|\xi| \in \bbS^{d-1}$, $\xi \in \R^d$, $\xi \neq 0$,  and
\[
A(\omega)=\sum_{j=1}^{d}\omega_{j}\bA^{j}= \begin{pmatrix}
\bthe\,\bp_{\rho}\bet\,\bbu\cdot \omega /\brho & \bthe\,\bp_{\rho}\bet\,\omega & 0 & 0 \\ \bthe\,\bp_{\rho}\bet \omega^{\top} & \brho\,\bthe\,\bet \bbu\cdot \omega \mathbf{I}_d & \bthe\,\bet\,\bp_{\theta}\omega^{\top} & \frac{1}{3}\bthe\,\bet\, \omega^{\top} \\ 0 & \bthe\,\bet\,\bp_{\theta}\omega & \brho\, \be_{\theta}\bet\,\bbu\cdot \omega & 0 \\ 0 & \frac{1}{3}\bthe\,\bet\, \omega & 0 & \bthe\,\bbu\cdot \omega/4
\end{pmatrix},
\]
\[
B(\omega)=\sum_{j,k=1}^{d}\omega_{j}\omega_{k}\bB^{jk}= \begin{pmatrix}
0 & 0_{1\times d} & 0 & 0 \\ 0_{d\times 1} & 0_{d} & 0_{d\times 1} & 0_{d\times 1} \\ 0 & 0_{1\times d} & 0 & 0 \\ 0 & 0_{1 \times d} & 0 & \frac{\bthe}{12\sigma_{s}}
\end{pmatrix},
\]
for $\omega\in \bbS^{d-1}$, and where $\bL$ is given by \eqref{L-b}. By the form of $\bL$ and the relation $\bet=\bthe^4$, it is easy to check that
\[
\mbox{ker}\big( \bL \big)=\mbox{span}\left( \left\lbrace (1,0_{1\times d},0,0),(0,\hat{e}_{j},0,0),(0,0_{1\times d},1/(4\bthe^3),1)\, : \, \, j=1,\ldots,d \right\rbrace \right),
\] 
while 
\[
\mbox{ker}\big( B(\omega) \big)=\mbox{span}\left( \left\lbrace (1,0_{1\times d},0,0),(0,\hat{e}_{j},0,0),(0,0_{1\times d},1,0)\, : \, j=1,\ldots,d \right\rbrace \right),
\]
which is independent of $\omega$. Thus we obtain
\[
\mbox{ker}\big( \bL \big)\cap \mbox{ker}\big( B(\omega) \big)=\mbox{span}\left( \left\lbrace (1,0_{1\times d},0,0),(0,\hat{e}_{j},0,0) \, : \, j=1,\ldots,d \right\rbrace \right).
\]

Let us take $\psi\in \mbox{ker}\big( \bL \big)\cap \mbox{ker}\big( B(\omega) \big) $, $\psi=(0,\hat{e}_{d},0,0)\neq 0$. Then for all directions $\omega=(\omega_{1},\omega_{2},\ldots,\omega_{d})\in\bbS^{d-1}$ such that $\omega_{d}=0$, which always exist as we are assuming $d\geq 2$, we obtain
\[
A(\omega)\psi= \begin{pmatrix}
0 \\ \brho\,\bthe\,\bet\, (\bbu\cdot\omega) \, \hat{e}_{d}^{\top}\\ 0 \\ 0
\end{pmatrix},
\quad
\mu \bA^0\psi = \begin{pmatrix}
0 \\ \mu\,\brho\,\bthe\,\bet\, \hat{e}_{d}^{\top} \\ 0 \\ 0
\end{pmatrix},
\]
so that
\[
\mu \bA^{0}\psi+ A(\omega)\psi= \begin{pmatrix}
0 \\ \brho\,\bthe\,\bet (\bbu\cdot \omega + \mu)\hat{e}_{d}^{\top} \\ 0 \\ 0
\end{pmatrix}= 0,
\]
as long as we take $\mu = -\bbu\cdot \omega$. The computations above show that the linear system \eqref{symm-linedsyst} does not satisfy the genuine coupling condition when $d\geq 2$. 
\end{proof}
\begin{remark}
Let us observe that the argument of the proof above does not apply to one space dimension ($d = 1$) because in that case $\omega=\pm 1$ and hence we cannot choose $\omega$ such that $\omega_{d}=0$.
\end{remark}

\section{Spectral analysis of the semigroup for small wave numbers}
\label{sptr-anal}

Next, we prove a result that was used along the proof of Lemma \ref{lin-dec-shr-lem}. The analysis is basically the same as that presented in \cite[Appendix A]{KY09}. However, a small modification must be made to take into account the dissipation due to viscosity mechanisms. Let us start by defining $\Psi(z):=\Gamma \Phi(z) \Gamma^{-1}$, where $\Gamma:=\big( \tiA^{0}\big)^{1/2}$ and $\Phi(z)$ is defined after \eqref{def-sem}:
\[
\Phi(z)= - \big( \tiA^{0} \big)^{-1}\left( \tiL + z \tiA^{1}-z^{2} \tiB \right).
\]
Then we consider $\Psi(0)=\Gamma \Phi(0) \Gamma^{-1}=-\Gamma^{-1}\tiL \Gamma^{-1}$, which is real symmetric and non-positive definite. Thus $\Psi(0)$ has an spectral representation of the form
\begin{equation}\label{A.1}
\Psi(0)= \sum_{j=1}^{r}\lambda_{j}\bar{\Pi}_{j},
\end{equation}
where the $\lambda_{j}'$s are the distinct $r$ eigenvalues of $\Psi(0)$ and $\bar{\Pi}_{j}'$s are the corresponding eigenprojections. The $\bar{\Pi}_{j}'s$ are real symmetric satisfying
\[
\sum_{j=1}^{r}\bar{\Pi}_{j} = I, \quad \bar{\Pi}_{j}\bar{\Pi}_{k}=\delta_{jk}\bar{\Pi}_{j},
\]
where $\delta_{jk}$ is the Kronecker delta. 
In \eqref{A.1} we sort out the eigenvalues such that $\lambda_{1}=0$ and $\lambda_{j}<0$ for $j=2,\ldots, r$. Then $\bar{\Pi}_{1}\R^{4}$  is the null space of $\Psi(0)$, which implies
\begin{equation}\label{A.1-1}
\bar{\Pi}_{1}\R^{4}= \Gamma \mathcal{M},
\end{equation}
where $\mathcal{M}$ is the null space of $\tiL$. 

As $\Psi(z)$ is a polynomial family of matrices depending on the complex parameter $z\in \C$, there is only a finite number of coalescing points in the complex plane; see Texier \cite[Proposition 1.3]{Tex18}. This implies that the coalescing points are isolated. 
Thus for $z$ close to $0$, but different of $0$, the eigenvalues of $\Psi(z)$ are of constant multiplicity, and we can write the spectral decomposition 
\begin{equation}\label{A.2}
\Psi(z) = \sum_{\ell =1}^{\bar{r}}\lambda_{\ell}(z)\Pi_{\ell}(z),
\end{equation}
where $\bar{r}$ is constant. In the above representation the eigennilpotent part of each eigenvalue is zero, as a consequence of $\Psi(z)$ being real symmetric (and hence diagonizable) for $z$ real and because of analytic continuation to $z$ complex. Moreover, one can easily show that $\lambda_{\ell}(z)$ and $\Pi_{\ell}(z)$ are analytic at $z=0$ (see, e.g., Liu and Zeng \cite[Lemma 6.8]{LiuZ97}), and the $\Pi_{\ell}(z)$ are real symmetric for $z$ real and they satisfy
\begin{equation}\label{A.3}
\sum_{\ell=1}^{\bar{r}}\Pi_{\ell}(z)=I, \quad \Pi_{\ell}(z)\Pi_{k}(z)=\delta_{\ell k}\Pi_{\ell}(z),
\end{equation}
for $\ell, k=1,2, \ldots, \bar{r}$.  Evaluating \eqref{A.2} at $z=0$ and using \eqref{A.1} we get
\[
\sum_{\ell=1}^{\bar{r}}\lambda_{\ell}(0)\Pi_{\ell}(0)= \sum_{j=1}^{r}\lambda_{j}\bar{\Pi}_{j}.
\]
Thus by the uniqueness of the spectral decomposition as the sum of a diagonalizable operator and a nilpotent one, we obtain that for each $j=1,2,\ldots,r$, there exists $n_{j}$ of the $\bar{r}$ eigenvalues $\lambda_{\ell}(z)$ such that $\lambda_{j}=\lambda_{\ell}(0)$. We rename these $n_{j}$ eigenvalues by $\lambda_{j\alpha}(z)$, with $\alpha=1,2,\ldots, n_{j}$, so that the spectral decomposition \eqref{A.2} can be rewritten as
\begin{equation}\label{A.4}
\Psi(z)= \sum_{j=1}^{r}\sum_{\alpha=1}^{n_{j}}\lambda_{j\alpha}(z)\Pi_{j\alpha}(z),
\end{equation}
while \eqref{A.3} becomes
\begin{equation}\label{A.4-1}
\sum_{j=1}^{r}\sum_{\alpha=1}^{n_j}\Pi_{j\alpha}(z)=I, \quad \Pi_{j \alpha}(z)\Pi_{j^{\prime}\alpha^{\prime}}(z)=\delta_{j j^{\prime}}\delta_{\alpha \alpha^{\prime}}\Pi_{j \alpha}(z).
\end{equation}
In addition, there holds
\begin{equation}\label{A.6}
\bar{\Pi}_{j}=\sum_{\alpha=1}^{n_{j}}\Pi_{j\alpha}(0),\quad \lambda_{j}=\lambda_{j\alpha}(0).
\end{equation}

In what follows we are going to derive some bounds on the $\lambda_{j\alpha}(z)'$s appearing in \eqref{A.4} for $z$ close to zero, specifically for $z=i\xi$ for $\xi\in \R$ small. In doing so we consider the cases $j=1$ and $j=2,\ldots, r$ separately. We start by the latter case, that is for $j=2,\ldots, r$. We know that
\[
\lambda_{j\alpha}(0)=\lambda_{j}<0,\,\, \mbox{for}\,\, \alpha=n_{1},\ldots, n_{j}.
\]
Thus by continuity of the eigenvalues (see \cite[Proposition 1.1]{Tex18}), there holds
\begin{equation}\label{re-lm-j>1}
\Re \lambda_{j\alpha}(i \xi) \leq -c_{1},
\end{equation}
for some uniform constant $c_{1}>0$ for $\vert \xi \vert \leq R_{1}$, for $j=2,\ldots,r$ and $\alpha=1,2,\ldots, n_{j}$, and some  $0< R_{1} \ll 1$.

For $j=1$, by the Equivalence Theorem \ref{Eq-Th} (observe that the eigenvalues of $\Psi(i\xi)$ are the same as those of the eigenvalue problem associated to the linear system \eqref{sys-z-lin}, which is genuine coupled), we have
\[
\Re \lambda_{1\alpha}(i\xi) \leq -c_{2}\frac{\xi^2}{1+\xi^2},
\]
for some uniform constant $c_{2}>0$, and for all $\xi\in \R$ and all $\alpha=1,2,\ldots, n_{1}$ (see statement (iv) of Theorem \ref{Eq-Th}). Thus for $\vert \xi \vert \leq R_{1}$ we can write 
\begin{equation}\label{re-lm-1}
\Re \lambda_{1\alpha}(i\xi) \leq -c_{3} \xi^2,
\end{equation}
for all $\alpha=1,\ldots,n_{1}$, and some other positive uniform constant $c_{3}$.

As the $\lambda_{j\alpha}(z)$ and $\Pi_{j\alpha}(z)$ are analytic at $z=0$, they can be written in the form
\begin{equation}\label{A.7}
\lambda_{j\alpha}(z)=\sum_{k=0}^{\infty}\lambda_{j\alpha}^{(k)}z^{k},\quad \Pi_{j\alpha}(z)=\sum_{k=0}^{\infty}\Pi_{j\alpha}^{(k)}z^{k}.
\end{equation}
The representation above together with \eqref{A.1-1}, \eqref{A.4-1} and \eqref{A.6} imply
\begin{equation}\label{A.8}
\Pi_{1\alpha}^{(0)}\R^4=\Pi_{1\alpha}(0)\R^{4}\subset \bar{\Pi}_{1}\R^4= \Gamma \mathcal{M},
\end{equation}
for $\alpha=1,2,\ldots, n_{1}$. 

We are ready to prove the following lemma.

\begin{lemma}\label{sptr-anal-lem}
Assume that $\hg(\xi)\in \mathcal{M}^{\bot}$, then there exists positive constants $c$, $C$ and $R$ such that
\[
\vert e^{t\Phi(i\xi)}\big( \tiA^{0} \big)^{-1}\hg(\xi) \vert \leq C e^{-ct}\vert \hg(\xi) \vert + C \vert \xi \vert e^{-c\xi^{2}t} \vert \hg(\xi) \vert,
\]
for $\xi \leq R$, and for all $t\geq 0$.
\end{lemma}
\begin{proof}
Using the spectral decomposition \eqref{A.4}, \eqref{A.1-1} and that $\Phi(z)=\Gamma^{-1} \Psi(z) \Gamma$, $\Gamma=\big(\tiA^{0} \big)^{1/2} $, we get
\begin{equation}\label{A.9}
\begin{aligned}
e^{t\Phi(i\xi)}\big( \tiA^{0} \big)^{-1}&\hg(\xi) = \Gamma^{-1}  e^{t\Psi(i\xi)}\Gamma^{-1}\hg(\xi) \\ &= \sum_{j=1}^{r}\sum_{\alpha=1}^{n_{j}} e^{\lambda_{j\alpha}(i\xi)t}\Gamma^{-1}\Pi_{j\alpha}(i\xi) \Gamma^{-1}\hg(\xi).
\end{aligned}
\end{equation}
We consider the cases $j=1$ and $j=2, \ldots, r$. For the latter one, thanks to the bound \eqref{re-lm-j>1} and the analiticity of the $\Pi_{j\alpha}(z)$ at $z=0$ we get
\begin{equation}\label{w-1}
\Big \vert \sum_{j=2}^{r}\sum_{\alpha=1}^{n_{j}} e^{\lambda_{j\alpha}(i\xi)t}\Gamma^{-1}\Pi_{j\alpha}(i\xi) \Gamma^{-1}\hg(\xi) \Big \vert \leq C_{1}e^{-c_{1}t}\vert \hg(\xi) \vert,
\end{equation}
for some uniform constants $C_{1}$, $c_{1}$ for $\vert \xi \vert \leq R_{1}$, with $c_{1}$ and $R_{1}$ such that \eqref{re-lm-j>1} holds.  

For $j=1$, we use the hypothesis $\hg(\xi)\in \mathcal{M}^{\bot}$ and \eqref{A.8} to obtain
\[
\< \Pi_{1\alpha}(0)\Gamma^{-1}\hg(\xi), \phi \> = \< \hg(\xi), \Gamma^{-1}\Pi_{1\alpha}(0)\phi \>=0, \quad \alpha=1,\ldots,n_{1},
\]
for all $\phi \in \R^4$, where we have used the fact that $\Gamma^{-1}$ and $\Pi_{1\alpha}(0)$ are real symmetric. Thus $\Pi_{1\alpha}^{(0)}\Gamma^{-1} \hg(\xi) = \Pi_{1\alpha}(0)\Gamma^{-1} \hg(\xi)=0$, so that
\[
\Pi_{1\alpha}(i\xi) \Gamma^{-1}\hg(\xi) = \sum_{k=1}^{\infty}\Pi_{1\alpha}^{(k)}(i \xi)^{k}\Gamma^{-1} \hg(\xi),
\]
for $\alpha=1,\ldots,n_{1}$. Then combining the expression above and the bound \eqref{re-lm-1} we are led to
\begin{equation}\label{w-2}
\Big \vert \sum_{\alpha=1}^{n_1} \Gamma^{-1}e^{\lambda_{1\alpha}(i\xi)t}\Pi_{1\alpha}(i\xi)\Gamma^{-1} \hg(\xi)  \Big \vert \leq C_{2} \vert \xi \vert e^{-c_{3}\xi^2t} \vert \hg(\xi) \vert,
\end{equation}
for some uniform constants $C_{2}$ and $c_{3}$ for $\vert \xi \vert \leq R_{1}$, with $c_{3}$ and $R_{1}$ such that \eqref{re-lm-1} holds. 

The proof concludes by combining \eqref{w-1} and \eqref{w-2}, and by taking $C=\max \{C_1, C_2 \}>0$, $c= \min \{c_1, c_3 \}>0$ and $R=R_{1}$.

\end{proof}

%


\def\cprime{$'\!\!$}

\end{document}